\documentclass[11pt]{amsart}
\usepackage{cite}
\usepackage{geometry}
\usepackage{amsmath,amssymb,graphicx,blkarray,amsthm} 
\usepackage{mathtools}

\usepackage{bbm}
\usepackage{color,subfig}
\usepackage[table]{xcolor}
\usepackage[colorlinks,linkcolor=blue,citecolor=blue,urlcolor=black,bookmarks=false,hypertexnames=true]{hyperref} 
\usepackage{float}
\usepackage{enumitem}
\usepackage{subdepth}

\usepackage{tikz}
\usetikzlibrary{calc,angles,positioning,quotes,decorations}
\usepackage{tkz-euclide}

\usepackage{verbatim}
\usepackage{tikz-cd}

\usepackage[all,cmtip]{xy}

\theoremstyle{plain}
\newtheorem{theorem}{Theorem}
\newtheorem{proposition}[theorem]{Proposition}
\newtheorem{lemma}[theorem]{Lemma} 
\newtheorem{corollary}[theorem]{Corollary}

\theoremstyle{definition}
\newtheorem{definition}[theorem]{Definition}
\newtheorem{remark}[theorem]{Remark}

\newtheorem{example}[theorem]{Example}

\newcommand{\Z}{\mathbb Z}
\newcommand{\N}{\mathbb N}

\newcommand{\ii}{{\mathrm{i}}}

\newcommand{\ee}{{\mathrm{e}}}

\begin{document}

\title[Torsion-free $S$-adic shifts and their spectrum]{Torsion-free $S$-adic shifts and their spectrum}

\author[\'A. Bustos-Gajardo]{\'Alvaro Bustos-Gajardo}
\address{
School of Mathematics and Statistics \\
The Open University\\
Walton Hall, Kents Hill, Milton Keynes\\
MK7 6AA, UK \\
\& Facultad de Matem\'aticas \\ 
Pontificia Universidad Cat\'olica de Chile\\
Edificio Rolando Chuaqui, Campus San Joaquín\\
Avda. Vicuña Mackenna 4860\\
Macul, Santiago, Chile
}
\email{alvaro.bustos-gajardo@open.ac.uk}

\author[N. Ma\~nibo]{Neil Ma\~nibo}
\address{
School of Mathematics and Statistics \\  
The Open University\\
Walton Hall, Kents Hill, Milton Keynes\\
MK7 6AA, UK \\ \&
Faculty of Mathematics\\
Bielefeld University\\
Postfach 10 01 31\\
33501, Bielefeld, Germany}
\email{cmanibo@math.uni-bielefeld.de}

\author[R. Yassawi]{Reem Yassawi}
\address{
School of Mathematics and Statistics  \\ 
The Open University\\
Walton Hall, Kents Hill, Milton Keynes\\
MK7 6AA, UK \\ \&
School of Mathematical Sciences, Queen Mary University of London\\
Mile End Rd, London E1 4NS}
\email{r.yassawi@qmul.ac.uk}

\date{}

\begin{abstract}
In this work we study $S$-adic shifts generated by sequences of morphisms that are constant-length.
	We call a sequence of constant-length morphisms  {\em torsion-free} if any prime divisor of one of the lengths is a divisor of infinitely many of the lengths.
	We show that torsion-free directive sequences generate shifts that enjoy the property of {\em quasi-recognizability} which can be used as a substitute for recognizability. Indeed quasi-recognizable directive sequences can be replaced by a rec\-og\-niz\-a\-ble directive sequence. With this, we give a finer description of the spectrum of shifts generated by torsion-free sequences defined on a sequence of  alphabets of bounded size, in terms of extensions of the notions of {\em height} and {\em column number}. We illustrate our results throughout with examples that explain the subtleties that can arise.
\end{abstract}



\maketitle

\section{Introduction}
In \cite{Dekking1977}, Dekking completed the works of  Martin \cite{martin-1971} and Kamae \cite{kamae-1972} to give a description of the spectrum of constant-length substitution shifts. These dynamical systems are defined by  one substitution $\theta\colon \mathcal A\rightarrow \mathcal A^{+}$ which is iterated repeatedly to generate a language and so a shift space. If we move away from this stationary setting, and instead  iterate a given  {\em directive} sequence $(\theta^{(n)})$ of morphisms  
	with $\theta^{(n)}\colon \mathcal A_{n+1}\rightarrow \mathcal A_n^{+}$,  we generate {\em $S$-adic shifts}; see Section  \ref{main_S_adic} for some background.
 In this article we study the spectrum of $S$-adic shifts generated by a sequence of constant-length morphisms defined on alphabets $(\mathcal A_n)$ of bounded size. We call  the cor\-re\-spond\-ing shifts constant-length $S$-adic shifts. Note that families of these shifts have been extensively studied. To begin with, they generalise constant-length substitution shifts, for which a huge literature exists. Also, {\em Toeplitz} shifts \cite{Williams, GjerdeJohansen2000} are often $S$-adic shifts generated by constant-length directive sequences. Finally, the constant-length property is very useful in generating interesting and tractable examples, such as  in \cite[Section 2]{Durand-2003} and \cite[Example 4.3]{BSTY-2019}.
	
	$S$-adic shifts have been studied, in one guise or another, over the last several decades. With various restrictions they can be interpreted as sym\-bol\-ic versions of cutting-and-stacking trans\-for\-ma\-tions of the unit interval, or Bratteli--Vershik dynamical systems. One important restriction is that of {\em recognizability}  (Definition  \ref{def:recog}), a property of the given directive sequence which gives some geometric structure to the symbolic space on which the dynamics operates. Although there are important families of recognizable directive  sequences,
recognizability is a property that is generally difficult to guarantee. 

On studying the results in  \cite{Dekking1977}, it is interesting to  note that in fact they do not entirely depend on recognizability. Enough geometric structure is obtained from the existence of a factor map onto a group rotation. To be precise, this group rotation is accessible, and non-trivial, because constant-length substitution shifts have an abundance of rational eigenvalues. This property extends to our setting. We define the strictly weaker notion of {\em quasi-recognizability} (Definition \ref{def:quasi-recognisability}), which requires the existence of a factor map to a rotation defined by the lengths of the morphisms, and we  show that it is  sufficient to deduce a good quantity of spectral information about our dynamical system. 
  If the morphisms we consider are injective on letters, then a quasi-recognizable directive sequence is recognizable (Lemma \ref {quasi-almost-recog}), but  Example \ref{ex:quasi_recognizable_but_not_recognizable} tells us  that quasi-recognizable shifts are not always recognizable. We will see below that this is not a serious obstacle.
  Next, 
   we  identify the large family of {\em torsion-free} directive sequences $(\theta^{(n)})$, namely, if $p$ divides the length of some $\theta^{(k)}$, then it must divide  the length of   $\theta^{(n)}$ infinitely often. We show in Theorem \ref{thm:dekking-sadic} that torsion-free sequences are quasi-recognizable. 
   We also include Remark \ref{rem:non-bounded}, which indicates that our results can be extended to a larger family of shifts that are defined on sequences of  alphabets of non-bounded cardinality.
 The example in \cite{BSTY-2019}, of a constant-length $S$-adic shift which is not recognizable, is not torsion-free; see Example~\ref{ex:quasirec-under-factor}. We phrase many of our results for  the class of torsion-free directive sequences;  however we note that they often hold for a larger class of quasi-recognizable directive sequences.

In  Theorem  \ref{thm: mef-c-to-one} we show that if a shift is  a somewhere finite-to-one extension of an odometer, then its continuous eigenvalues must all be rational. We  apply this result in Corollary \ref{cor:MEF-torsion-free}, to show that  torsion-free $S$-adic shifts only have ra\-tion\-al eigenvalues, and we can quantify them quite precisely. In addition to the natural eigenvalues that come from  the lengths of the morphisms, there is essentially one other eigenvalue. As in the substitutional case, we call this the {\em height}. In other words, the maximal equicontinuous factor of a torsion-free $S$-adic shift is a group extension of the tiling factor by a cyclic group, and furthermore  we see in Theorem \ref{thm:fully-essential-constant-coboundarybis} that this finite cyclic group is arithmetically orthogonal to the tile lengths in a strong sense.  As a result we have an $S$-adic version of Cobham's theorem in Theorem \ref{cor:cobham}, where we show that a torsion-free $S$-adic shift  cannot be generated by two  sequences of morphisms whose lengths are not compatible.

Of particular interest are torsion-free constant-length shifts whose max\-i\-mal equicontinuous factor space is a torsion-free group: These are the  torsion-free shifts whose height is trivial. As for constant-length substitutions, we show in Theorem \ref{thm:height} that a torsion-free constant-length shift has a {\em pure base}, in that it is a constant height suspension  over another constant-length $S$-adic shift which has trivial height and which has the same sequence of  morphism lengths. Thus we see that the family of shifts generated by  torsion free $S$-adic shifts is quite robust; this robustness recurs in results we describe below.  Also,  this family is easily seen to be closed under the taking of factors (Corollary \ref{cor:factor-quasi}). However, quasi-recognizable shifts are not closed under factoring; see Example \ref{ex:quasirec-under-factor}.

Until this point, quasi-recognizability is sufficient for our needs. To com\-bi\-na\-to\-ri\-al\-ly characterise the height, or to define the notion of a {\em column number}, which gives us further spectral in\-for\-ma\-tion, we need more than quasi-recognizability. But
it turns out that from quasi-recognizability we can manufacture recognizability. Precisely, given a quasi-recognizable constant-length shift $(X_{\boldsymbol \theta},\sigma)$, generated by the directive sequence ${\boldsymbol \theta}$, we can modify ${\boldsymbol \theta}$  to obtain a recognizable directive sequence  $\boldsymbol {\widehat \theta}$ such that $X_{\boldsymbol \theta}=X_{\boldsymbol {\widehat\theta}}$
 (Theorem \ref{thm:sadic-recognizable}). Here also, the sequence of morphism lengths is not changed. Therefore  whenever we need to, we can  assume  that a quasi-recognizable $S$-adic shift is generated by a recognizable directive sequence.

 Thus we can work with a  recognizable representation of a torsion-free shift, which is available to us by Theorem \ref{thm:sadic-recognizable}.
 Recall that the height of a constant-length substitution can also be characterised combinatorially: it is the largest number that is coprime to the length of the substitution, and which divides the greatest common divisor of the return times to a letter in the alphabet. 
  We define in Section \ref{sec:combinatorial-height} a combinatorial height for torsion-free directive sequences, and we show in Theorem \ref{thm:combinatorial-height} that it equals the height.
 We also completely extend the notion of a {\em column number}. The column number of an $S$-adic shift gives us important spectral information, in particular whether  it is almost automorphic in Corollary \ref{cor:odom-factor-k},
 and whether the shift has discrete or mixed spectrum in Propositions \ref{prop:discrete-spec} and \ref{prop:mixed-spec}.

	In  Section \ref{sec:preliminaries}, we set up notation and define background concepts. In Section \ref{sec:recog}, we introduce the property of quasi-recognizability, relate it to recognizability, and show that quasi-recognizability is enjoyed by a large class of the $S$-adic shifts that we study, and also that they have a recognizable representation. In Sections \ref{sec:height}  and \ref{sec:combinatorial-height} we formulate equivalent notions of height, and show that constant-length $S$-adic systems with non-trivial height are  suspensions over a pure base system with trivial height. In Section \ref{sec:column-number}, we define the notion of a column number, which then enables us to give a  finer description of the spectrum of a constant-length $S$-adic shift in Section \ref{sec:spectrum}.

\section{Preliminaries}\label{sec:preliminaries}
	Given a sequence $(q_n)_{n\geq 0}$ of natural numbers,
we work with the product of sets \[\Z_{(q_n)} \coloneqq \prod_{n} \Z/q_n\Z,\] which can be endowed with  the group operation  given by coordinate-wise addition with carry.
For a detailed exposition of equivalent definitions of $\Z_{(q_n)}$, we refer the reader to \cite{Downarowicz2005}. 
We write elements $(z_n)$ of $\Z_{(q_n)}$ as left-infinite sequences $\cdots z_2z_1$, where $z_n\in \Z/q_n\Z$,
so that addition in $\Z_{(q_n)}$ has the carries 
propagating to the left as is usual in $\Z$. If $(q_n)$ is the constant sequence $q_n=q$, then  $\Z_{(q_n)}= \Z_q$ is the classical ring of $q$-adic integers. 
We can also define $\Z_{(q_n)}$ as the inverse limit  $\Z_{(q_n)} =  \varprojlim \Z/p_n\Z $ of cyclic groups, where $p_0=1$ and $p_n\coloneqq  q_0\cdots q_{n-1}$. 
The integers can be injected into  $\Z_{(q_n)}$, via $1= (\cdots, 0,0,1)$, and $-1 = (\cdots, q_3-1, q_2-1,q_1-1)$.
Note that if almost all $q_n=1$, then $\Z_{(q_n)}$ is finite. For this reason, in this article we will assume that $q_n\geq 2$ infinitely often; given such a sequence we can always multiply consecutive terms of the sequence  to obtain $q_n\geq 2$ for each $n$.

Endowed with the product topology over the discrete topology on each $\Z/q_n\Z$, the group $\Z_{(q_n)}$ is a compact metrisable  topological group, 
where the unit 1, as defined in the previous paragraph, is a topological generator.
This topology is metrisable: two points $z, z' \in \Z_{(q_n)}$ are within $\varepsilon$ if they agree on a large initial portion, i.e. $z_i = z_i'$ for $1\leq i \leq n$, for $n=n(\varepsilon)$. With the above notation, an {\em odometer} is a dynamical system $( Z, +1)$ where 
 $ Z=\Z_{(q_n)}$ for some sequence $(q_n)$. 

 Let $\mathcal{A}$ be a finite set of symbols, also called an {\it alphabet}, and let $\mathcal{A}^\mathbb{Z}$ denote the set of  two-sided infinite sequences over $\mathcal{A}$. 
 Similar to the topology we define on  $\Z_{(q_n)}$, we equip  $\mathcal{A}^\mathbb{Z}$ with the metrisable product topology.
 In this work, we consider shift dynamical systems, or {\em shifts} $(X,\sigma)$, where $X$ is a closed $\sigma$-invariant set of $\mathcal A^\mathbb Z$ and $\sigma\colon \mathcal A^\mathbb Z\to \mathcal A^\mathbb Z$ is the left shift map  $(x_n)_{n\in\mathbb{Z}} \mapsto (x_{n+1})_{n\in\mathbb{Z}}$. 
    We use letters $x,y,$ etc, to denote points in the  two-sided shift space $\mathcal{A}^\mathbb{Z}$.   In some cases we will discuss  non-invertible, or one-sided, shifts  $(\tilde{X},\sigma)$, where $\tilde{X}\subset \mathcal{A}^\mathbb{N}$  stands for the set of one-sided,  or right-infinite sequences over $\mathcal A$.   For $\mathbb L=\mathbb N $ or $\mathbb Z$, we use $[w_0\cdots w_n]$  to denote the cylinder set $\{      x\in \mathcal A^\mathbb L : x_0\cdots x_n = w_0\cdots w_n\}.$ A  shift is {\em minimal} if it  has no non-trivial closed shift-invariant subsets.   
    We say that $x \in \mathcal{A}^\mathbb{L}$ is \emph{periodic} if $\sigma^k(x) = x$ for some $k \ge 1$, \emph{aperiodic} otherwise.
The~shift $(X,\sigma)$ is said to be \emph{aperiodic} if each $x \in X$ is aperiodic.   For basics on
  continuous and measurable dynamics see Walters \cite{Walters}.

  Given a finite alphabet  $\mathcal{A}$, let  $\mathcal{A}^*$ be the free monoid of all (finite) words over $\mathcal{A}$ under the operation of concatenation, and let $\mathcal{A}^+$  be the set of all non-empty words over $\mathcal{A}$. We let $\lvert w\rvert$ denote   the length of a finite word~$w$ and  let $\lvert\mathcal A\rvert$ denote the cardinality of  the set $\mathcal A$. A {\em subword}   of  a word or a sequence $x$  is    a finite word $x_{[i,j)}$, $i \le j$, with $x_{[i,j)} \coloneqq  x_i x_{i+1}\cdots x_{j-1}$. 
   A {\em language} is a collection of words in $\mathcal{A}^*$. 
The \emph{language~$\mathcal{L}_x$} of $x = (x_n)_{n\in \mathbb Z} \in \mathcal{A}^\mathbb{Z}$ is the set of all its subwords. 
The language~$\mathcal{L}_X$ of a one- or two-sided shift $(X,\sigma)$ is the union of the languages of all $x \in X$; it is closed under the taking of subwords and every word in $\mathcal{L}_X$ is left- and right-extendable  to a word in $\mathcal{L}_X$. 
Conversely,  a language~$\mathcal{L}$ on~$\mathcal{A}$ which is closed under the taking of subwords and such that each word is both left- and right-extendable defines a one- or two-sided shift $(X_\mathcal{L}, \sigma)$, where $X_\mathcal{L}$ consists of   the set of points all of whose subwords belong to~$\mathcal{L}$, so that $\mathcal L_{X_\mathcal{L}} = \mathcal L$.   
The two-sided shift defined by $\mathcal L$ is the {\em natural extension} of the one-sided shift defined by $\mathcal L$.

Let $\mathcal A$ and $\mathcal B$ be finite alphabets, and let  $\theta\colon \mathcal{A}\to \mathcal{B}^+$ be a map; it extends to a  morphism $\theta\colon \mathcal{A}^*\to \mathcal{B}^*$, also  called a \emph{substitution} if $\mathcal A=\mathcal B$. Note that we assume that  the image of any letter is a non-empty word.  Using concatenation, we extend $\theta$ to act on
 ~$\mathcal{A}^{+}$,
 ~$\mathcal{A}^\mathbb{N}$ and~$\mathcal{A}^\mathbb{Z}$.  If $\theta$ is a substitution, then the finiteness of $\mathcal A$ guarantees that  $\theta$-periodic points, i.e., points $x$ such that $\theta^k(x)=x$ for some  positive $k$, exist.
The \emph{incidence matrix} of the morphism $\theta$ is the $\lvert\mathcal{B}\rvert \times \lvert\mathcal{A}\rvert$ matrix~$M_\theta=(m_{ij})$ with $m_{ij}$ being the number of occurrences of~$i$ in $\theta(j)$.    A substitution is  \emph{primitive} if its incidence matrix admits  a  power  with positive entries. 
Given a substitution $\theta\colon \mathcal{A} \rightarrow \mathcal{A}^+$, the  language   $\mathcal{L}_{\theta}$    defined by  $\theta$ is
\[
\mathcal{L}_{\theta} = \big\{w \in \mathcal{A}^*:\, \mbox{$w$ is a subword of $\theta^n(a)$ for some  $a\in\mathcal{A}$ and $n\in \mathbb N$}\big\}.
\]
If $\theta$ is primitive then each word in $\mathcal{L}_{\theta}$ is left- and right-extendable, and $\mathcal{L}_{\theta}$ is closed under the taking of subwords, so 
$X_{\theta}\coloneqq X_{\mathcal L_\theta}$.
We call $(X_\theta, \sigma)$ a \emph{substitution shift}. 
The substitution $\theta$ is \emph{aperiodic} if $(X_\theta,\sigma)$ has no shift-periodic points, and $\theta$
is a  (constant-)length $\ell$ substitution if
 $\lvert\theta(a)\rvert=\ell$ for each $a\in \mathcal A$.

\subsection{$S$-adic shifts } \label{main_S_adic}
We  recall basic definitions concerning $S$-adic shifts. They 
are obtained by replacing the iteration of a single substitution by the iteration of a sequence of  morphisms.
  In this article, we restrict to the case of an $S$-adic shift defined over a sequence of alphabets $\mathcal A_n$ of bounded cardinality, but see Remark~\ref{rem:non-bounded} below. 

Let $\boldsymbol{\theta} = (\theta^{(n)})_{n\ge0}$ be a sequence of morphisms with $\theta^{(n)}\colon \mathcal{A}_{n+1}\to \mathcal{A}_n^+$; we call   
 $\boldsymbol{\theta}$ a {\em directive sequence}. In this article, we assume that we work with a {\em  sequence of alphabets of bounded size}, i.e., the sequence $(\lvert\mathcal A_n\rvert)_{n\geq 0}$ is bounded. As proofs of results for directive sequences defined on a sequence of  alphabets of bounded size are a notational modification of proofs where $\mathcal A=\mathcal A_n$ for each $n$, we often work with a fixed alphabet, and we lose no generality in our statements.

 For $N\geq 1$ and $0\leq n<N$, let 
\[
\theta^{[n,N)} = \theta^{(n)} \circ \theta^{(n+1)} \circ \cdots \circ \theta^{(N-1)},
\]
 we shall call a word of the form $\theta^{[0,n)}(a)$, for some $a\in\mathcal{A}_n$, an $n$-\emph{supertile}, or an $n$-\emph{th order supertile}.  Note that every $n$-supertile is a concatenation of $\frac{p_{n}}{p_{m}}$ $m$-supertiles whenever $1\leq m<n$. Similarly, any element from $X_{\boldsymbol{ \theta}}$ may be seen as an infinite concatenation of $n$-supertiles, up to a shift.

For $n\geq 0$, define  $\bar{\mathcal{L}}^{(n)}=     \bar{\mathcal{L}}_{\boldsymbol{\theta}}^{(n)}$ as 
\begin{align*}
\bar{\mathcal{L}}^{(n)}& = \big\{w \in \mathcal{A}_n^*:\, \mbox{$w$ is a subword of $\theta^{[n,N)}(a)$ for some $a \in\mathcal{A}_N$, and $N>n$}\big\}.
\end{align*}

 Define
\[ X^{(n)}=X^{(n)}_{\boldsymbol{\theta}}\coloneqq \{x\in \mathcal A_n^{\Z}: \mbox{ for each $k\leq \ell, x_{[k,\ell)}\in \bar{\mathcal{L}}_{\boldsymbol{\theta}}^{(n)}$ } \}\]
and note that $  \mathcal{L}^{(n)}      \coloneqq  \{ w:  w \mbox{ appears as some subword in some } x\in X^{(n)}    \}$
 is generally a proper subset of   $\bar{\mathcal{L}}^{(n)}$. 
 We call  $(\mathcal{L}^{(n)})_{n\geq 0}$ the languages associated to~$\boldsymbol{\theta}$. In some situations, see for example Proposition \ref{prop:two-one-sided-continuous}, it will be more convenient to work with the one sided shift
 \[ \tilde{X}^{(n)}=\tilde{X}^{(n)}_{\boldsymbol{\theta}}\coloneqq \{x\in \mathcal A_n^{\N}: \mbox{ for each $0\leq k\leq \ell, x_{[k,\ell)}\in \bar{\mathcal{L}}_{\boldsymbol{\theta}}^{(n)}$ } \}.\]

We say that ~$\boldsymbol{\theta}$ is \emph{primitive} if for each $n\ge0$ there is an $N>n$ such that  the incidence matrix $M_{[n,N)}\coloneqq  M_{\theta^{(n)}}M_{\theta^{(n+1)}}\cdots M_{\theta^{(N-1)}}$ of $\theta^{[n,N)}$ is a positive matrix. Under the assumption of primitivity for $\boldsymbol{\theta}$,
each word in $ \mathcal{L}^{(n)}$ is left- and right-extendable. 
 If $\boldsymbol{\theta}$ is primitive, then each $(X^{(n)},\sigma)$ is minimal for all~$n$
 \cite[Lemma~5.2]{Berthe-Delecroix}.  We only work with  directive sequences such that $(X^{(n)},\sigma)$ is minimal for all~$n$  in this article.

We set $X_{\boldsymbol{\theta}} = X_{\boldsymbol{\theta}}^{(0)}$ and call $(X_{\boldsymbol{\theta}},\sigma)$ the \emph{$S$-adic shift} generated by the directive sequence ~$\boldsymbol{\theta}$.

\begin{definition}[Constant-length directive sequences]\label{def:directive-families} Let $\boldsymbol{\theta}= (\theta^{(n)})_{n\geq 0}$ be a directive sequence. We say that    $\boldsymbol{\theta}$ is 
a {\em constant-length} directive sequence on $(\mathcal A_n)$, if each $\theta^{(n)}$ is a constant-length morphism. We say that ${\boldsymbol \theta}$  has {\em length sequence} $(q_n)_{n\geq 0}$ if for each $n$, $\theta^{(n)}$ is of length $q_n$.
\end{definition}

Note that there are no constraints on the  sequence $(q_n)$ and it  is  not assumed  constant or bounded.

\subsection{Eigenvalues and equicontinuous factors}\label{sec:spectrum-background}

Let $\boldsymbol{\theta}$ be a constant-length directive sequence with shift $(X_{\boldsymbol{\theta}},\sigma)$. 
A complex number $\lambda\in S^1$ is a {\em continuous eigenvalue} if there exists a continuous function $f\colon X_{\boldsymbol{\theta}}\rightarrow\mathbb C$ with $f\circ \sigma= \lambda f$.
We abuse terminology and say that  $\lambda$ is a {\em rational} eigenvalue if  $\lambda= \ee^{2 \pi \ii p/q} $ is a root of unity.

Let $\mu$ be a $\sigma$-invariant Borel probability measure on $X_{\boldsymbol{\theta}}$ and consider the Hilbert space $L^{2}(X,\mu)$. We say that $\lambda$ is a {\em measurable eigenvalue}  if there exists a non-zero $f\in L^{2}(X,\mu)$ such that $f\circ \sigma=\lambda f$. In general, not all measurable eigenvalues admit continuous eigenfunctions.

The following proposition, with a  proof in \cite{BCY-2022}, will be convenient for us, sometimes allowing us to work with one-sided shifts where  our arguments have less cumbersome notation.
\begin{proposition}\label{prop:two-one-sided-continuous}
Let $(\tilde{X},\tilde{\sigma})$ be a one-sided minimal  shift, and  let $(X,\sigma)$ be its natural extension. Then
$\lambda$ is a continuous eigenvalue for  $(\tilde{X},\tilde{\sigma})$ if and only if $\lambda$ is a continuous eigenvalue for $(X,\sigma)$.
  If  $\tilde{\mu}$ is a shift invariant measure on $(\tilde{X},\tilde{\sigma})$ and  $\mu$ is the corresponding measure on $(X,\sigma)$, then   $\lambda$ is a measurable eigenvalue  for   $(\tilde{X},T, \tilde{\sigma})$ if and only $\lambda$ is a measurable eigenvalue for  $(X,\sigma, \mu)$.
\end{proposition}

A topological dynamical system  $(Z,S)$ is called {\em equicontinuous} if the family $\{S^n : n\in\Z\}$ is equicontinuous. 
A minimal equicontinuous system $(Z,S)$  must be a  rotation on a compact monothetic topological group, i.e., there exists an element $g\in Z$ such that the subgroup generated by $g$ is dense,  and the homeomorphism  is  $S(z) = z+g$. Such a group is always abelian  and we will write the group operation additively.

If $(Z,+g)$ is equicontinuous and there is a  factor map $\pi \colon(X,\sigma)\rightarrow (Z,+g)$, we say that $(Z,+g)$ is an equicontinuous factor of $(X,\sigma)$. Any $\Z$-action $(X,\sigma)$ admits a maximal equicontinuous factor $\pi\colon(X,\sigma)\to (Z,+g)$.  
This equicontinuous factor must be maximal, i.e.,  any other  equicontinuous factor of $(X,\sigma)$ factors through $(Z,+g)$. The maximal equicontinuous factor encodes all continuous eigenvalues of $(X,\sigma)$. There are two ways to see this, both of which we will use and both of which are described in \cite{ABKL-2015}. First,  the maximal equicontinuous factor $(Z,+g)$ encodes the continuous eigenvalues of $X$ in the sense that  the  Pontryagin dual $\hat{Z}$ of $Z$ may be interpreted as the subgroup of $S^1$ generated by all continuous eigenvalues of $(X,\sigma)$.
Second, we can  put an equivalence relation on $X$ where $x\sim y$ if and only if $f(x)=f(y)$ for every continuous eigenfunction $f$. Then it can be shown that $\sigma$ induces an equicontinuous map on $X/{\sim}$, and that each continuous eigenfunction on $X$ translates to an eigenfunction on $X/{\sim}$.
The shifts that we study in this article  only have rational continuous eigenvalues; see Corollary \ref{cor:MEF-torsion-free}.  In this case the relation $\sim$ equals $\Lambda\coloneqq  \bigcap_{n\geq 1} \Lambda_n$, where $\Lambda_n$ is  defined and used in Section \ref{sec:recog}.

\subsection{Limit words}

\begin{definition} [Limit word]\label{def:limit-word} Let $\boldsymbol{\theta}$ be a directive sequence.
We say that $u=u^{(0)}\in X_{\boldsymbol{\theta}}$ is a (two-sided) {\em limit word} if there exists a sequence $(u^{(n)})_{n\geq 0}$
with $u^{(n)}\in X^{(n)}$, and  $\theta^{(n)}(u^{(n+1)})= u^{(n)}$ for each $n$.
Similarly, 
 $u=u^{(0)}\in \tilde{X}_{\boldsymbol{\theta}}$ is a  (one-sided) {\em limit word} if there exists a sequence $(u^{(n)})_{n\geq 0}$
with $u^{(n)}\in \tilde{X}^{(n)}$, and  $\theta^{(n)}(u^{(n+1)})= u^{(n)}$ for each $n$.
 \end{definition}

We say that a word $w$   is {\em essential} for the directive sequence $ \boldsymbol{ \theta} $ if it occurs in $\mathcal{L}^{(n)}$ for  infinitely many $n$. An essential word is {\em fully essential} if it occurs in  $\mathcal{L}^{(n)}$  for each $n$.
{\em Telescoping} a directive sequence $(\theta^{(n)})_{n\geq 0}$ means taking a sequence $(n_k)_{k\geq 1}$ and considering instead the directive sequence $(\tilde{\theta}^{(k)})$ where $\tilde {\theta}^{(0)} = \theta^{[0,n_1)}$ and $\tilde{\theta}^{(k)} = \theta^{[n_{k},n_{k+1})}$ for $k\geq 1$. Telescoping a directive sequence does not change the dynamics, i.e.,    $X_{\boldsymbol{\theta}} =  X_{\boldsymbol{\tilde\theta}}$.  However telescoping can change the nature of the directive sequence, for example, a  directive sequence can be {\em finitary}, i.e., it is  chosen from a finite set of morphisms, but it has telescopings that are not finitary. This was a concern in \cite{BCY-2022}, where finitary directive sequences were studied. In this article, we do not need to assume that our directive sequences are finitary, nor do we put any constraints on the morphisms in $\boldsymbol {\theta}$, only that they are defined on alphabets of bounded size. Therefore we can telescope freely, and if $ab$ is essential, we can assume that it is fully essential, i.e., that  it appears in $\mathcal L^{(n)}$ for each $n$.

We  link our definition of a limit word to that in \cite{BCY-2022}. There,
a two-sided limit word is defined as 
 $u\coloneqq  \lim_{k} \theta^{[0,n_k)}(a)\cdot  \theta^{[0,n_k)}(b)$ for some essential word $ab$ that belongs to $  \mathcal{L}^{(n_k)}   $ for each $k$.
If the word $ab$ is essential, then as we can assume  that $ab$ is fully essential, we have $u\coloneqq  \lim_{n} \theta^{[0,n)}(a)\cdot  \theta^{[0,n)}(b)$. Recall that our directive sequence is defined on alphabets of bounded size. By further telescoping if needed, we can assume that 
 $\theta^{[0,n)}(b)$  shares a common prefix with  $\theta^{[0,n+1)}(b)$ of  increasing length, and $\theta^{[0,n)}(a)$  also shares a common suffix with $\theta^{[0,n+1)}(a)$ of increasing length. The sequence of finite words $(\theta^{[0,n)} (a)\cdot   \theta^{[0, n)} (b))_{k\geq 1}$ converges to a bi-infinite sequence $u$ in $X_{\boldsymbol{\theta}}$, and for each $k$, the sequence of finite words $(\theta^{[k,n)} (a)\cdot   \theta^{[k, n)} (b))_{k\geq 1}$ converges to $u^{(k)}\in \mathcal L^{(k)}$ and $\theta^{(k)}(u^{(k+1)})=u^{(k)}$. Thus a limit word in the sense of \cite{BCY-2022} is a limit word as in Definition \ref{def:limit-word}. Conversely, if we have a limit word as in Definition \ref{def:limit-word}, we take a sequence $n$ such that  $(u^{(n)}_{-1} u^{(n)}_{0})_n$ is a constant sequence $ab$, and then $\lim_{n} \theta^{[0,n)} (ab)$ converges to the limit word $u$.

\section{Recognizability and quasi-recognizability}\label{sec:recog}
We  first start with   a notion  which expresses the idea of performing a  ``desubstitution''.
\begin{definition}[Dynamic recognizability, $\theta$-representations  and rec\-og\-niz\-a\-ble directive sequences] \label{def:recog}
Let $\theta\colon \mathcal{A} \to \mathcal{B}^+$ be a morphism and $y \in \mathcal{B}^\mathbb{Z}$.
If $y = \sigma^k \theta(x)$ with $x=(x_n)_{n\in \mathbb Z}
\in \mathcal{A}^\mathbb{Z}$, and  $0 \leq k < \lvert\theta(x_0)\rvert$,  then we say that $(k,x)$ is a \emph{(centred) $\theta$-representation} of~$y$. 
For $X \subseteq \mathcal{A}^\mathbb{Z}$, we say that the $\theta$-representation $(k,x)$ \emph{is in~$X$} if $x \in X$.

Given $X \subseteq \mathcal{A}^\mathbb{Z}$ and $\theta\colon \mathcal{A} \to \mathcal{B}^+$, we say that $\theta$ is \emph{recognizable in~$X$} if each $y \in \mathcal{B}^\mathbb{Z}$ has at most one centered $\theta$-representation in~$X$. 
A~directive sequence $\boldsymbol{\theta}$ is \emph{recognizable at level~$n$} if $\theta^{(n)}$ is recognizable in~$X^{(n+1)}$. 
The sequence~$\boldsymbol{\theta}$ is \emph{recognizable} if it is recognizable at level~$n$ for each $n \ge 0$.
\end{definition}

Note that the notion of recognizability of a shift is incompatible with the existence of shift-periodic points in that shift \cite{BSTY-2019}.

\begin{definition}[Quasi-recognizability]\label{def:quasi-recognisability}
Let ${\boldsymbol \theta}$ be a  constant-length directive sequence with length sequence $(q_n)$, where each $(X_{\boldsymbol{\theta}}^{(n)},\sigma)$ is minimal.
  If there is an equicontinuous factor map $\pi_{\rm tile}\colon (X_{\boldsymbol {\theta}}, \sigma) \rightarrow (\Z_{(q_n)}, +1)$ then  we say that 
${\boldsymbol \theta}$ is {\em quasi-recognizable}, and we call $\pi_{\rm tile}$ the {\em tiling factor map}.
\end{definition}

\begin{remark}\label{rem:limit-word-zero}
If  $\pi \colon (X_{\boldsymbol {\theta}}, \sigma) \rightarrow (\Z_{(q_n)}, +1)$ is an  equicontinuous factor map, then there exists a factor map  $\pi_{\rm tile}\colon (X_{\boldsymbol {\theta}}, \sigma) \rightarrow (\Z_{(q_n)}, +1)$  which maps one limit word to zero in $\Z_{(q_n)}$. Hence we will always  assume that $\pi_{\rm tile}$ maps some limit word to 0.
\end{remark}

Recall that if $\boldsymbol{\theta}$ is primitive, then each $(X_{\boldsymbol{\theta}}^{(n)},\sigma)$ is minimal.
We remark that the existence of a (surjective) factor map  $\pi_{\rm tile}\colon (X_{\boldsymbol {\theta}}, \sigma) \rightarrow (\Z_{(q_n)}, +1)$ forces $ (X_{\boldsymbol {\theta}}, \sigma)$ to be infinite. The assumption of  minimality implies that if  $(X_{\boldsymbol {\theta}}, \sigma)$ is quasi-recognizable then
$ X_{\boldsymbol{\theta}}^{(0)}  $ is aperiodic. Throughout this article, we work with quasi-recognizable directive sequences. Hence, we implicitly assume that our shifts contain no shift-periodic points.

In light of Proposition \ref{prop:two-one-sided-continuous}, as we have defined it, the property of being quasi-recognizable is not sensitive to the one- or two-sided setting. In other words,
$(X_{\boldsymbol \theta},\sigma)$ is quasi-recognizable if and only if $(\tilde{X}_{\boldsymbol \theta},\tilde\sigma)$ is quasi-rec\-og\-niz\-a\-ble.
  This is contrary to recognizability, where there are shifts that are two-sided but not one-sided recognizable \cite{Mosse-1992}.

 \begin{lemma}\label{quasi-almost-recog}
Let ${\boldsymbol \theta}$ be a constant-length directive sequence with length se\-quence $(q_n)$, and  such that each $(X_{\boldsymbol{\theta}}^{(n)},\sigma)$ is minimal.
If ${\boldsymbol \theta}$ is recognizable, then it is  quasi-recognizable. Conversely, if ${\boldsymbol \theta}$ is quasi-recognizable and each morphism $\theta^{(n)}$ is injective on letters, then ${\boldsymbol \theta}$ is recognizable.
\end{lemma}
\begin{proof} Suppose ${\boldsymbol \theta}$ is recognizable. This direction appears in \cite[Remark 6.1]{BCY-2022} but we include a proof here. For each $n$, consider the $\sigma^{p_n}$-cyclic partition 
\[\left\{\sigma^{j}\left(\bigcup_{a\in \mathcal A_n}\left[\theta^{[0,n)}(a)\right]\right), j=0, \cdots ,p_{n}-1\right\}\] of $X_{\boldsymbol \theta}$. This defines a $\sigma$-tower where we choose the  base to be $\bigcup_{a\in \mathcal A_n}\big[\theta^{[0,n)}(a)\big]$. Define $\pi_n\colon X_{\boldsymbol \theta}\rightarrow \Z/p_n\Z$ to be $\pi_n(x) = j_n$ if and only if $x$ belongs to $ \sigma^{j_n}\left(\bigcup_{a\in \mathcal A_n}\big[\theta^{[0,n)}(a)\big]\right)$. As $j_{n+1}\equiv j_n \pmod{p_n}$, the maps $(\pi_n)$ define a factor map $\pi_{\rm tile}\colon X_{ \boldsymbol \theta     } \rightarrow \Z_{(q_n)}$. Further the assumption on the base of each partition implies that any limit word is mapped by $\pi_{\rm tile}$ to $0$.

Conversely, suppose that we have a factor $\pi_{\rm tile}\colon X_{\boldsymbol \theta}\rightarrow \Z_{(q_n)}$. We can assume, by rotation if needed, that  it  maps  a  limit word to $0$; see  Remark \ref{rem:limit-word-zero}. Fix such a limit word $u$.  Define, for $n\in \N$ and $0\leq j \leq p_n -1$,  $U_{n,j}\coloneqq  \{ x : \pi_{\rm tile}(x)_n = j\}$. We have 
$\sigma^k(u) \in  U_{n,j}$ if and only if $k\equiv j \pmod{p_n}$, and $\sigma^k(u)$ has an $n$-supertile  $\theta^{[0,n)}(a)$ with support $[-j,p_n -j-1]$. By the assumption of minimality of  $X_{\boldsymbol \theta}$, we also have $x\in U_{n,j}$ if and only if $x$ has an $n$-th order supertile   with support $[-j,p_n -j-1]$. \\
In other words, the existence of such a $\pi$ implies that given $x\in X_{\boldsymbol \theta}$, we have complete information about the indices at which an $n$-th 
order supertile begins, for each $n$. The result follows.
\end{proof}

 It remains to give some general condition which guarantees that $\boldsymbol{\theta}$ is quasi-recognizable. Note that a constant-length $S$-adic shift is not necessarily recognizable; see \cite[Example 4.3]{BSTY-2019}. However that example does not satisfy the following definition.
	\begin{definition}[Torsion-free directive sequences]\label{def:torsion-free}
	Let   $\boldsymbol{\theta}$ be a constant-length directive sequence with length sequence $(q_n)$ where $q_n>1$ infinitely often,  where each $(X_{\boldsymbol{\theta}}^{(n)},\sigma)$ is minimal, and such that  $ X_{\boldsymbol{\theta}}^{(0)}  $ is aperiodic.
 If each prime $p$ which divides some $q_k$ divides  $q_n$ infinitely often, then we say that  $\boldsymbol{\theta}$ is {\em torsion-free}.
	\end{definition}

	The following theorem tells us that  the shift generated by a torsion-free $S$-adic shift factors onto a torsion-free odometer. We will see in Corollary 
	\ref{cor:MEF-torsion-free} that this odometer is not always the maximal equicontinuous factor, and that  the latter may have a torsion factor.

	\begin{theorem}\label{thm:dekking-sadic}
	Let $\boldsymbol{\theta}$ be a torsion-free  directive sequence, defined on a  sequence of bounded alphabets.
		Then  ${\boldsymbol \theta}$ is quasi-recognizable.
		\end{theorem}

	To prove Theorem \ref{thm:dekking-sadic} we use the notation and a modified procedure from \cite{Dekking1977}. The following notions originate in \cite{Gottschalk-Hedlund-1955}. A non-empty closed $\sigma^k$-invariant subset of $X$ is {\em $\sigma^k$-minimal} if it contains no proper closed $\sigma^k$-invariant sets.
	A cyclic partition  $\{ X_1, X_2, \cdots , X_m\}$ of size $m$ of $X$ is a partition where $\sigma(X_i)= X_{(i+1) \bmod m}$ for each $i$. 
	The cyclic partition of size $m$ is {\em $\sigma^n$-minimal} if each partition element is
	 $\sigma^n$-minimal. Note that as we work with Cantor spaces,  elements of a $\sigma^n$-minimal partition must be clopen.
	Define $\gamma(n)$ to be the cardinality of a cyclic $\sigma^n$-minimal partition. This partition is unique up to cyclic permutation, so we can define the accompanying equivalence relation $\Lambda_n$ on $X$ where $\Lambda_n(x)=\Lambda_n(y)$ if and only if $x, y$ belong to the same 
	$\sigma^n$-minimal partition element. Let $\nu_p(n)$ denote the $p$-adic valuation of $n$.
	
	\begin{lemma}\label{lem:partition} Let $(X,\sigma)$ be a minimal shift. Then
	\begin{enumerate}
	\item $1 \leq \gamma(n) \leq n$ and $\gamma(n) \mid n$,
	\item $\Lambda_{\gamma(n)}= \Lambda_n$ and $\gamma(\gamma(n))=\gamma(n)$,
	\item if $m\mid n$ then $\Lambda_n\subset \Lambda_m$, and if $\gamma(n)=n$ then $\gamma(m)=m$,
	\item if $\gamma(n)>1$ then there exists $m\mid n$ such that $\gamma(m)=m$,
	\item if $(m,n)=1 $ then $\Lambda_{mn}= \Lambda_m \cap \Lambda_n$ and $\gamma(mn)= \gamma(m)\gamma(n)$, and 
	\item if $p$ is prime, $\gamma(p)<p$ and  $(p_n)$ is a sequence such that  $\nu_p(p_n)\uparrow \infty$ then $\gamma (p_n)/p_n \rightarrow 0$. 
	 \end{enumerate}
	\end{lemma}
	
	\begin{proof}
	 The statements \textbf{(1)}-\textbf{(5)} are in \cite[Lem.~3]{Dekking1977}.
	We prove only the last statement. Write $p_n = p^{k_n}s_n$ where $(p,s_n)=1$. By Statement (5), 
	we have $\gamma(p_n)= \gamma ( p^{k_n}) \gamma(s_n)$. We claim that $\gamma(p^a)<p$ for each $a\in \N$. For suppose that $\gamma(p^a)= p^b$ with $1\leq b\leq a$. Then $\gamma(p^b) = \gamma^2(p^a)= \gamma(p^a)= p^b$ by (2). Now $\gamma(p^b) = p^b$ and (3) implies that $\gamma(p)=p $, a contradiction to our assumption. Thus, since $\nu_p(p_n)\uparrow \infty$, 
	\[  \frac{ \gamma(p_n)}{p_n} = \frac{ \gamma (p^{k_n}) \gamma(s_n) }{p^{k_n} s_n} \leq   \frac{ \gamma (p^{k_n})  }{p^{k_n} } < \frac{ p  }{p^{k_n} } \rightarrow 0. \]
	\end{proof}
	
	\begin{proof}[Proof of Theorem \ref{thm:dekking-sadic}] 
	We will prove the theorem in the case  where the directive sequence is defined on a single alphabet $\mathcal A$, to ease notation. The proof is the same if ${\boldsymbol \theta}$ is
	defined on a sequence of alphabets, as long as their size remains bounded.

	Let $u=u^{(0)}\in X_{\boldsymbol \theta}$ be a limit word, i.e., there is a sequence $(u^{(n)})_{n\geq 0}$ with $u^{(n)}\in X^{(n)}$ and $u^{(n)}= \theta^{(n)} (u^{(n+1)})$. Let  $(q_n)$ be the length sequence of  ${\boldsymbol \theta}$.
	We claim that there are at most 
	$\lvert\mathcal A\rvert + \lvert\mathcal A\rvert^2 (\gamma(p_n)-1)$ distinct words of length $p_n$ in $u$ and therefore in  the language of $X_{\boldsymbol \theta}$.
	
	Let $X_{n} = \theta^{[0,n)}(X^{(n)})$; then minimality of $X^{(n)}$ implies that   $X_{n}$ is a $\sigma^{p_n}$-minimal set in $X_{\boldsymbol \theta}$. 
		Since $u= \theta^{[0,n)}(u^{(n)})$, 
		we have $u\in X_{n}$.
	Therefore $\sigma^{k \gamma(p_n)}u \in X_n\subset \bigcup_{a\in \mathcal A} \big[\theta^{[0,n)}(a)\big]$ for each $k\in \Z$. Thus
	$u$ is composed of overlapping words of the form $\theta^{[0,n)}(a)$, of length $p_n$, spaced at intervals $\gamma(p_n)$. Since there are at most
	$\lvert\mathcal A\rvert$ words of the form $\theta^{[0,n)}(a) $ and at most $\lvert\mathcal A\rvert^2$ words of the form $\theta^{[0,n)}(ab)$, we have proved our claim, as any word $w$ of length $p_n$ in $u$ is either an $n$-supertile or overlaps two adjacent $n$-supertiles; in the second case the first $n$-supertile may appear at most $\gamma(p_n)-1$ positions before $w$ in $u$.

	If $\gamma(p_n)<p_n$ for some $n$, then by Lemma \ref{lem:partition}, $\lim_n \gamma (p_n)/p_n = 0$. We can therefore find an $n$ such that
	$\lvert\mathcal A\rvert + \lvert\mathcal A\rvert^2 (\gamma(p_n)-1)< p_n$. But then there are fewer than $p_n$ words of length $p_n$ in $u$, and so $u$ is periodic, a contradiction to the hypothesis that $ X_{\boldsymbol{\theta}}^{(0)}  $ is aperiodic.
	
	Therefore $\gamma(p_n)=p_n$ for each $n$ and  we can  construct $\pi_{\rm tile}\colon X_{\boldsymbol \theta}\rightarrow \Z_{(q_n)}$. Namely, each relation $\Lambda_{p_n}$ defines a cyclic tower of height $p_n$ so we have a map $\pi_n\colon X\rightarrow \Z/p_n\Z$, and  we can choose a cyclic permutation of the tower so that its base $B_n$ contains $u$. In this way $\pi(u)=0$, and also $\pi_{n+1}(x) \bmod{p_{n}} = \pi_n(x)$
	 (here
	 we are implicitly using the inverse-limit form of $\Z_{(q_n)}$). 
   Also, $\sigma^{j p_n}(u)\in B_n$ for each $j$, and these are the only shifts of $u$ that belong to $B_n$. Finally by minimality,  $x\in B_n$ if and only if
	$$x= \lim_{\ell\rightarrow \infty} \sigma^{j_\ell p_n} (u)=  \lim_{\ell\rightarrow \infty} \sigma^{j_\ell p_n}  \theta^{[0,n)}(u^{(n)}) =  \lim_{\ell\rightarrow \infty}  \theta^{[0,n)}(\sigma^{j_\ell}u^{(n)}),  $$
	so that $ B_n =  \theta^{[0,n)}(X^{(n)})$. Similarly, $\sigma^j(B_n) = \sigma^j (\theta^{[0,n)}(X^{(n)}))$ for $0<j<p_n$. 
	
	 Note that we can take $\pi_{\rm tile}(u)=0$ whenever $u$ is a limit word as follows. Each $\Lambda_{p_n}$ defines a tower of height $p_n$ which is unique up to cyclic rotation and we have taken the base to contain each limit word $u$. By minimality, $x$ belongs to the base only if  $x\in  \theta^{[0, n+1)}(X^{(n+1)})$. This also means that $\pi_{\rm tile}(x)=0$ only if $x\in \bigcap_{n\ge 0}  \theta^{[0, n)}(X^{(n)})$, i.e. if $x$ is a limit word.
	 \end{proof}
	 
	 \begin{remark}\label{rem:non-bounded}
	 Note that the requirement that we   work with a sequence of bounded alphabets is used in exactly one place.  Let $p_n=q_0 q_1\dots q_{n-1}$. The torsion-free assumption on the directive sequence gives us $\lim_n \gamma (p_n)/p_n = 0$. The boundedness of alphabet cardinalities allows us to  find an $n$ such that
	$\lvert\mathcal A_n\rvert + \lvert\mathcal A_n\rvert^2 (\gamma(p_n)-1)< p_n$. Therefore, given a sequence $(p_n)$ we can relax the condition on the alphabet sizes, requiring only  that  $\lim_n \frac{\lvert\mathcal A_n\rvert}{p_n} =0$. With this condition, we can  generalise the statement of Theorem \ref{thm:dekking-sadic}. Namely,  suppose   that $\theta^{(n)}:\mathcal A_{n+1}\rightarrow \mathcal A_n^{q_n}$, and that ${\boldsymbol \theta}=(\theta^{(n)})$ is torsion free. If  
	 \[\lim_n \frac{\lvert\mathcal A_n\rvert}{q_0\dots q_{n-1}} =0,\] 
	 then $\boldsymbol{\theta}$ is quasi-recognizable. We state this as a remark as the emphasis in this paper is on bounded alphabet sequences of morphisms; but this implies that many of our results in this paper that hold for torsion-free sequences on bounded alphabets extends to this family.
	
	\end{remark}
	
	Suppose that $\boldsymbol{ \theta}$ is torsion-free and $\tau\colon X_{\boldsymbol{ \theta}}\to Y$ is a radius-$0$ sliding block map, i.e., a {\em code}, onto an infinite shift space $Y$. Then, by considering the directive sequence $\boldsymbol{ \theta}'\coloneqq (\tau\circ\theta^{(0)},\theta^{(1)},\dotsc)$, we see that $\boldsymbol{\theta}'$ is a torsion-free directive sequence whenever $X_{\boldsymbol{\theta}'}$ is infinite. 
	
	Next, suppose that $\tau \colon X_{\boldsymbol{ \theta}}\to Y$ has left radius $l$ and right radius $r$. We modify the directive sequence $(\theta^{(n)} )$, replacing each $\theta^{(n)}$ by its {\em $l+r+1$-sliding block presentation of  $X_{\boldsymbol{ \theta}}$} as follows.	  Given a length-$\ell$ substitution, $\theta$, we define $\eta$ on the alphabet consisting of words of length $l+r+1$ that belong to $\mathcal L_\theta$. Now if $(a_1,\dots ,a_{l+r+1})\in \mathcal L_\theta$, and $\theta(a_1,\dots ,a_{l+r+1})= b_1 \dots b_{(l+r+1)\ell }$, define
		\[\eta ( (a_1,\dots ,a_{l+r+1}))  := (b_1, \dots ,b_{\ell} ) (b_2, \dots , b_{\ell +1} ) \dots (b_\ell, \dots, b_{2\ell-1}); \]
	it is straightforward to show that $(X_\theta,\sigma)$ and $(X_{\eta},\sigma)$ are topologically conjugate; see for example \cite[Section 5.4]{queffelec}.
		In the $S$-adic setting, we can follow a similar procedure. Namely
	  if we replace the sequence 
	$(\theta^{(n)} )$ with $(\eta^{(n)} )$,  defining  $\eta^{(n)} $ on the alphabet of legal words from $\mathcal L^{(n+1)}$ of length $l+r+1$, then as we have not modified the substitution lengths, the new directive sequence ${\boldsymbol \eta}$ is also torsion-free, and defined on a bounded sequence of alphabets if  ${\boldsymbol \theta}$ is. One can also show that $(X_{\boldsymbol{ \theta}},\sigma)$ is topologically conjugate to $(X_{\boldsymbol{ \eta}},\sigma)$. Furthermore the sliding block code $\tau$ can be transferred to $\tau_\eta: X_{\boldsymbol{ \eta}}\rightarrow Y$, where $\tau_\eta$ is a code and $\tau( X_{\boldsymbol{ \theta}}   ) = \tau_\eta(X_{\boldsymbol{ \eta}})$.  In other words, $\tau( X_{\boldsymbol{ \theta}}   ) $ is generated by the directive sequence $(\tau\circ\eta^{(0)},\eta^{(1)},\dotsc)$.  
	
	Thus we can assume that 
	 any topological factor $(Y,\sigma)$ of $(X_{\boldsymbol \theta}, \sigma)$, with ${\boldsymbol \theta}$ torsion-free,  can be given by such a code, and we have the following. 
		Note that it is a far easier result to prove than the corresponding result for automatic sequences in \cite{Mullner-Yassawi-2021}.
	
	\begin{corollary}\label{cor:factor-quasi}	
		Let $\boldsymbol{\theta}$ be a torsion-free  directive sequence, 
		defined on a  sequence of bounded alphabets.
			Then any infinite factor of $(X_{\boldsymbol \theta}, \sigma)$ is  a torsion-free S-adic shift.	\end{corollary}

		Let $\boldsymbol{ \theta}=(\theta^{(j)})_{j\geq 0}$ and $\boldsymbol{\eta}=(\eta^{(j)})_{j\geq 0}$ 
		be two directive sequences sharing the same length sequence $(q_j)_{j \geqslant 0}$. If the shift spaces $X_{\boldsymbol{\theta}}$ and $X_{\boldsymbol{\eta}}$ are conjugate, we can see by composing the conjugacy map with the relevant  tiling factor map, if it exists, that  quasi-recognizability is a conjugacy invariant for $\mathcal{S}$-adic shifts. A natural follow-up question is thus whether a factor map between two	$\mathcal{S}$-adic shifts preserves quasi-recognizability. Below we see that this is not the case.
		\begin{example}\label{ex:quasirec-under-factor}
		Consider the two substitutions $\alpha$ and $\beta$ given by 
		\begin{align*}
			\alpha\colon A &\mapsto 00 &  \beta\colon A & \mapsto ACABA \\
			B &\mapsto 01 & B &\mapsto ACAAA \\
			C &\mapsto 10 & C &\mapsto AAABA
		\end{align*}
and the directive sequence $\boldsymbol{\alpha} = (\alpha, \beta, \beta,
		\beta, \ldots)$. This is a known example \cite{BSTY-2019} of a non-recognizable
		directive sequence, as each $x \in X_{\boldsymbol{\alpha}}$ can be written in
		two ways as $\alpha (y)$ or $\sigma \circ \alpha (z)$ for some $y, z \in
		X_{\beta}$. It cannot be quasi-recognizable either. For,  as both $\alpha$
		and $\beta$ are injective, quasi-recognizability would imply
		recognizability by Lemma~\ref{quasi-almost-recog}.
		
		However, we can easily convert this directive sequence into another that
		engenders a recognizable shift, by replacing $\alpha$ with the following
		substitution:
		\begin{align*}
			\bar{\alpha}\colon A &\mapsto 00 \\
			B &\mapsto 01 \\
			C &\mapsto \bar{1}0
		\end{align*}
		The new directive sequence $\bar{\boldsymbol{\alpha}} = (\bar{\alpha},
		\beta, \beta, \beta, \ldots)$ is easily seen to be rec\-og\-niz\-a\-ble. Indeed, any
		point $x \in X_{\bar{\boldsymbol{\alpha}}}$ must contain a $\bar{1}$
		somewhere, which necessarily is the start of a supertile $\bar{1} 0$; thus, the parity of the index $j$ where the symbol $\bar{1}$ is found determines whether $x$ is of the form $\bar{\alpha} (y)$ or $\sigma \circ \bar{\alpha} (y)$, for some $y \in X_{\beta}$. By injectivity of $\bar{\alpha}$, this $y$ is uniquely determined (and is thus a desubstitution of $x$). As $X_\beta$ is substitutive, recognizability is guaranteed from then on.
		
		The radius-$0$ code $f$ whose local function is given by $0\mapsto 0$ and  $1,\bar{1} \mapsto 1$ is a natural factor map $X_{\bar{\boldsymbol{\alpha}}} \to X_{\boldsymbol{\alpha}}$. This provides an example of a $S$-adic shift with a recognizable directive sequence and with a factor that is not even quasi-recognizable.

	\end{example}

	We end this section with a key result which will be useful later, in Sections \ref{sec:combinatorial-height}  and \ref{sec:column-number}. It tells us that 
		 we can turn a non-injective, quasi-rec\-og\-niz\-a\-ble directive sequence into a recognizable directive sequence, while keeping the top $S$-adic shift $X^{(0)}$ fixed.
		In other words, if we need recognizability, then we can manufacture it, provided that we have quasi-recognizability.

	\begin{theorem} \label{thm:sadic-recognizable}
		Let $\boldsymbol{\theta}$ be a quasi-recognizable  directive sequence 
	defined on a sequence of bounded alphabets.
	Then there exists  a recognizable directive sequence $\boldsymbol{\widehat\theta}$, 
	defined on a sequence of bounded alphabets, 
   such that $X_{  \boldsymbol{\theta}  }= X_{\boldsymbol{\widehat \theta}}$. 
		\end{theorem}
 \begin{proof}For ease of notation we assume that  the original directive sequence $\boldsymbol{\theta}$ is defined on the alphabet $\mathcal A$.
  We construct the sequence  $\boldsymbol{\widehat\theta}$ one morphism at a time.

   If $\theta^{(0)}$ is injective on letters, then we set $ \widehat\theta^{(0)}=\theta^{(0)}$. Otherwise, 
 we introduce an equivalence relation on $\mathcal A$ where $a\sim a'$ if $\theta^{(0)}(a) = \theta^{(0)}(a')$. By assumption there are $k<\lvert\mathcal A\rvert$ equivalence classes for $\sim$. 
 Define $\tau_1\colon \mathcal A\rightarrow \mathcal B_1$ where $\lvert \mathcal B_1 \rvert =k$ and where
 $\tau_1(a)=\tau_1(a')$ if and only if $a\sim a'$. Now define $\widehat\theta^{(0)} \colon \mathcal B_1\rightarrow \mathcal A^{q_0}$ by $\widehat\theta^{(0)}(b)= \theta^{(0)}(a)$ for any $a\in \tau_1^{-1}(b)$;  $\widehat\theta^{(0)}$ is well defined and injective on letters. Set $\eta^{(1)}\coloneqq  \tau_1\circ \theta^{(1)}$, then $\eta^{(1)}\colon \mathcal A\rightarrow \mathcal B_1^{(q_1)}$. Note that by construction, $\lvert\widehat\theta^{(0)}\rvert=q_0$ and $\eta^{(1)}=q_1$, and the tiling map $\pi_{\rm tile} \colon X^{(0)}\rightarrow   \Z_{(q_n)}$ does not change. Also,   $X^{(0)}_{ \boldsymbol{\theta}      } = X^{(0)}_{ ( \widehat\theta^{(0)}, \eta^{(1)}, \theta^{(2)}, \dotsc)}.$

 We can thus replace $  \boldsymbol{\theta}   $ with the directive sequence $   (\widehat\theta^{(0)}, \eta^{(1)}, \theta^{(2)}, \ldots)$; it still generates  $X^{(0)}_{ \boldsymbol{\theta}      }$. 
  If $\eta^{(1)}$ is injective on letters, we set 
 $\widehat\theta^{(1)} \coloneqq \eta^{(1)}$.
   Otherwise define $\tau_2 \colon \mathcal A\rightarrow \mathcal B_2 $ 
  by $\tau_2(a)=\tau_2(a')$ if and only if $\eta^{(1)}(a) = \eta^{(1)}(a')$,
    and  define  $\widehat\theta^{(1)}\colon \mathcal B_2\rightarrow \mathcal B_1^{(q_1)}$ 
    as $\widehat\theta^{(1)}(c) = \eta^{(1)} (a)$ for any $a\in \tau_2^{-1}(c)$;  as before $\widehat\theta^{(1)}$ is well defined and injective on letters. Define $\eta^{(2)} \coloneqq \tau_2 \circ \theta^{(2)}$; then $\eta^{(2)}$ is a morphism $\mathcal A\rightarrow \mathcal B_2^{( q_2)}$.  Now replace the directive sequence $   (\widehat\theta^{(0)}, \eta^{(1)}, \theta^{(2)}, \dotsc)$ with $( \widehat\theta^{(0)},   \widehat\theta^{(1)},  \eta^{(2)}, \theta^{(3)}, \dotsc)$;  as before the sequence of lengths remains the same, as does the tiling map $\pi_{\rm tile} \colon X^{(0)}\rightarrow   \Z_{(q_n)}$, and $X^{(0)}_{\boldsymbol{\theta}} = X^{(0)}_{(\widehat\theta^{(0)},\,  \widehat\theta^{(1)}, \,\eta^{(2)},\, \theta^{(3)}, \dotsc)}.$ 
\begin{figure}[h]
	\includegraphics{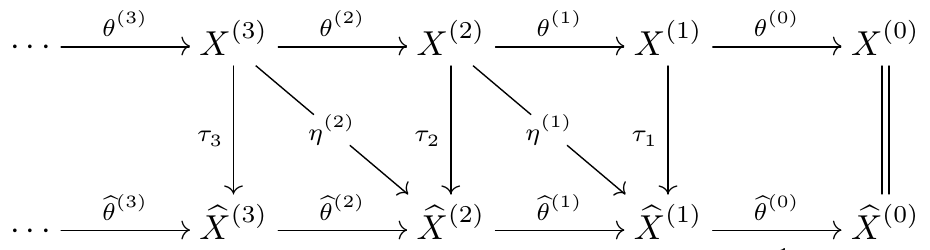}
	\caption{How to obtain an injective sequence of morphisms $\boldsymbol{\widehat\theta}$ from the original directive sequence $\boldsymbol{\theta}$. }	\label{fig:injectivization_of_dir_seq}
\end{figure} 

  We continue recursively in this way: see Figure \ref{fig:injectivization_of_dir_seq}.
  It is straightforward to see now that the directive sequence $\boldsymbol{ \widehat\theta} = (\widehat{\theta}^{(0)},\widehat{\theta}^{(1)},\widehat{\theta}^{(2)},\dotsc)$ has the desired properties. \end{proof}

\begin{example} \label{ex:quasi_recognizable_but_not_recognizable}
		In this example, we show that there exist quasi-recognizable directive sequences which are not  recognizable. Consider the substitutions $\theta,\vartheta$ and $\varrho$ given by
	\begin{align*}
		\theta\colon 0 &\mapsto 011 &  \vartheta\colon 0 &\mapsto 0\bar{1}1 & \varrho\colon 0 &\mapsto 011 \\
		1 &\mapsto 001 & 1 &\mapsto 0\bar{0}1 & 1 &\mapsto 001, \\
		& & \bar{0} &\mapsto \bar{0}1\bar{1} & \bar{0} &\mapsto 011 \\
		& & \bar{1} &\mapsto \bar{0}0\bar{1}, & \bar{1} &\mapsto 001,
	\end{align*}
	and consider the directive sequence $\boldsymbol{\alpha} = (\varrho,\vartheta,\vartheta,\vartheta,\dotsc)$. By definition, this is a torsion-free sequence on $\mathcal A=\{0,1\}$, and we can see that the shift $X_{\boldsymbol{ \alpha}}$ can be written as:
	\[X_{\boldsymbol{ \alpha}} = \varrho(X_\vartheta)\cup\sigma\circ\varrho(X_\vartheta)\cup\sigma^2\circ\varrho(X_\vartheta), \]
	where $X_\vartheta$ is the corresponding substitutive subshift.

	The net effect of applying $\varrho$ on a fixed point of $\vartheta$ is just the removal of the bars above the symbols.
		 Therefore $X_{\boldsymbol{ \alpha}}=X_\theta$, and 
		 		 $\boldsymbol{ \alpha}$  is quasi-recognizable. However,  $1.\bar{0}$ and $\bar{1}.0$ are valid seeds for distinct fixed points $x^{(1)} = \vartheta^\infty(1.\bar{0})$ and $y^{(1)} = \vartheta^\infty(\bar{1}.0)$ of $X_\vartheta = X^{(1)}$, which satisfy $\varrho(x^{(1)})= \varrho(y^{(1)})$, and thus the fixed point $\theta^\infty(1.0)\in X_\theta = X^{(0)}$ has two possible de\-sub\-sti\-tu\-tions in $X^{(1)}$. Hence, $\boldsymbol{ \alpha}$
				 				  is not recognizable.

	Applying Theorem \ref{thm:sadic-recognizable}
	 to the sequence $\boldsymbol{ \alpha}$, we note that $\varrho=\theta\circ\tau$, where $\tau$ is defined by $0,\bar{0}\mapsto 0$ and $1,\bar{1}\mapsto 1$; hence, $\widehat{\alpha}^{(0)} = \theta$, with the second substitution $\alpha^{(1)} = \vartheta$ being replaced by $\eta^{(1)} =\tau\circ\vartheta = \varrho$. Thus,
	  we obtain $\widehat{\boldsymbol{ \alpha}} = (\theta,\theta,\theta,\dotsc)$ as the injectivisation of $\boldsymbol{ \alpha}$, which is consistent with the observation that $X_{\boldsymbol{ \alpha}}=X_\theta$. Also, $\boldsymbol{ \widehat\alpha}$ is recognizable \cite{Mosse-1992}. 
\end{example}

\section{Height and the pure base}\label{sec:height}

	If a minimal  shift $(X,\sigma)$ is a somewhere one-to-one extension of an equicontinuous system, then that system must be the maximal equicontinuous factor of $(X,\sigma)$ \cite{Williams, Downarowicz2005}. This is no longer true if $(X,\sigma)$ is a somewhere finite-to-one extension of an odometer; see Example~\ref{ex:extension_not_product}. However  its maximal eq\-ui\-con\-tin\-u\-ous factor remains an odometer, as follows. We note that Basti\'an Espinoza has obtained a similar result via a different method \cite{Espinoza}.
	
	\begin{theorem}\label{thm: mef-c-to-one}
	Let $(X,\sigma)$ be a minimal shift. If $(X,\sigma)$ is a somewhere finite-to-one extension of an odometer $(\Z_{(q_n)}, +1)$, then the maximal equicontinuous factor of  $(X,\sigma)$ must be an odometer which is a rotation on a group extension of $\Z_{(q_n)}$ by a finite cyclic group $\mathbb{Z}/h\mathbb{Z}$.
	\end{theorem}

We can actually obtain Theorem~\ref{thm: mef-c-to-one} as a consequence of a more general result, without the need for substantial modifications of the proof. Indeed, we can prove that:

\begin{theorem}\label{thm:generalized-mef-c-to-one}
	Let $(X,T)$ be a minimal topological dynamical system and suppose that $(X,T)$ factors onto a group rotation $R_\alpha$ defined over a monothetic group $G$, where the cyclic subgroup generated by $\alpha\in G$ is dense in $G$. If the factor map $\pi_G\colon (X,T)\to (G,R_\alpha)$ is somewhere finite-to-one, then the maximal equicontinuous factor of $(X,T)$ is a group rotation $(Z,R_\beta)$ over a monothetic group $Z$ which is a finite extension of $G$.
\end{theorem}

\begin{proof} Let the finite $c$ be such that $(X,T)$ is a somewhere $c$-to-one extension of $(G,R_\alpha)$. Suppose $(Z,R_\beta)$ is the maximal equicontinuous factor of $(X,T)$. Due to the maximality of $Z$, we have the following commutative diagram of factor maps and group actions:
	\[
	\xymatrix{ X \ar[rrr]^\sigma \ar@{->>}[dd]_{\pi_{G}} \ar@{->>}[dr]^{\pi_{\rm MEF}} & & &  X \ar@{->>}[dd]^{\pi_{G}}\ar@{->>}[dl]_{\pi_{\rm MEF}} \\
		& Z \ar[r]^{+\beta} \ar@{-->>}[dl]^{\pi_{\rm ind}} & Z\ar@{-->>}[dr]_{\pi_{\rm ind}}\\
		G \ar[rrr]^{+\alpha} & & & G }
	\]
	where the somewhere $c$-to-one map $\pi_{G}\colon X\to G$ is $\pi_{G} = \pi_{\rm ind} \circ \pi_{\rm MEF}$. Hence, the following relationship holds:
	\[
	\pi_{G}^{-1}(\{g\}) = \bigcup_{z\in\pi_{\rm ind}^{-1}(\{g\})}\pi_{\rm MEF}^{-1}(\{z\}),
	\]
	where the union is disjoint. Since there exists a $g\in G$ with $\lvert \pi_{\rm G}^{-1}(\{g\}) \rvert = c$, then we must have $\lvert\pi_{\rm MEF}^{-1}(\{z\})\rvert \le c$ for every $z \in \pi_{\rm ind}^{-1}(\{g\})$, with equality if, and only if, $g$ has a single preimage in $Z$.
	
	We note that $f(z) = \pi_{\rm ind}(z) - \pi_{\rm ind}(0_Z)$ is a group homomorphism. Indeed, as $\pi_{\rm ind}$ is a factor map that commutes  the $\mathbb{Z}$-actions, one must have:
	\[
	\pi_{\rm ind}(g + \beta) = \pi_{\rm ind}(g) + \alpha \implies f(g + \beta) = f(g) + \alpha,\text{ for every }g\in Z,
	\]
	where $\beta$ and $\alpha$ generate dense cyclic subgroups of $Z$ and $G$, respectively. We identify these subgroups with $\mathbb{Z}$ in both cases. Since $f(0_Z) = 0_G$, it follows that $f(n\cdot \beta) = n\cdot \alpha$ for each  $n\in\mathbb{N}$.
	
	For any $h\in Z$, there is a sequence  of integers $\{h_n\}_{n\in\mathbb{N}}$ with $h_n \cdot \beta\to h$. As $\pi_{\rm ind}$, and so 
	$f$, are 
	  continuous then:
	\begin{align*}
		f(g + h) &= \lim_{n\to\infty} f(g + h_n\cdot \beta)\\
		&= \lim_{n\to\infty} f(g) + h_n\cdot \alpha \\
		&= f(g) + \lim_{n\to\infty} f(h_n\cdot \beta) \\
		&= f(g) + f(h).
	\end{align*}
	As $f$ is a group homomorphism, $f(g) = f(h) \iff g - h \in \ker(f)$. Thus:
	\begin{align*}
		g \in \pi_{\rm ind}^{-1}\left(\{\pi_{\rm ind}(h)\}\right) &\iff \pi_{\rm ind}(g) = \pi_{\rm ind}(h) \\
		&\iff f(g) = f(h) \\
		&\iff g - h \in \ker(f) \\
		&\iff g \in h + \ker(f).
	\end{align*}
	As $\pi_{\rm ind}$ is surjective, every $g \in G$ is of the form $\pi_{\rm ind}(h)$ for some $h\in Z$. Hence every fibre of $\pi_{\rm ind}$ is a coset of $\ker(f)$ and thus has cardinality $r = \lvert \ker(f) \rvert$. Since $\pi_{\rm tile}$ is somewhere finite-to-one, we may take some $g\in G$ such that $\lvert \pi^{-1}_{\rm tile}(\{g\})| = c$. The decomposition of $\pi_G^{-1}(\{g\})$ into fibres of $\pi_{\rm MEF}$ shown above allows us to see that this set is a finite, disjoint union of $r$ non-empty fibres $\pi^{-1}_{\rm MEF}(\{z\})$, and thus $r \le c$; in particular, $r$ must be finite.	By the first isomorphism theorem, $Z/\ker(f) \cong G$, and the result follows. \end{proof}

\begin{remark}
	The proof above is mostly group-theoretic in nature, with the key dynamical property used being that the group rotation $R_\alpha$ is minimal in $G$ and thus the latter group is monothetic, as the orbit of $0$ is a dense cyclic subgroup. This same argument, with minor modifications, applies to more general group actions $T\colon H\times X\to X$ (where $H$ is a sufficiently well-behaved group, such as $\mathbb{Z}^d$) where there exists a finite-to-one factor map $\pi\colon(Z,T,H)\to(G,(R_h)_{h\in H},H)$ onto a topological group $G$ for which there is a monomorphism $\iota\colon H\hookrightarrow G$ such that $\iota(H)$ is dense in $G$ and $R_h(g)=g+\iota(h)$ for all $h\in H$, that is, $G$ is a topological completion of $H$ in the same way as in which a monothetic group is a topological completion of $\mathbb{Z}$.
\end{remark}

Thus most  of Theorem~\ref{thm: mef-c-to-one} is a direct consequence of Theorem~\ref{thm:generalized-mef-c-to-one}; all that remains is the observation that $\ker(f)$ is a finite cyclic group $\Z/h\Z$ for some $h$. This is a consequence of the interpretation of the Pontryagin dual of $Z$ as the subgroup of $S^1$ generated by all continuous eigenvalues of the shift space $(X,\sigma)$, see Remark  \ref{rem:pontryagin_interpretation}, however, it may also be derived from the density of the orbit of $\beta$ in $Z$, using the fact that for every $k\in\ker(f)$ there is some $N_k\in\N$ such that $N_k\cdot\beta$ is arbitrarily close to $k$, and obtaining a single generator for $\ker(f)$ via arithmetic on the $N_k$'s and equicontinuity.

\color{black}

\begin{remark}\label{rem:pontryagin_interpretation}
	As discussed in Section \ref{sec:spectrum-background}, the  Pontryagin dual $\hat{Z}$ of the maximal equicontinuous factor may be interpreted as the subgroup of $S^1$ generated by  eigenvalues of $(X_{\boldsymbol \theta}, \sigma)$. Pontryagin duality and the proof above thus imply that $\hat{\mathbb{Z}}_{(q_n)} = \langle\{\ee^{2\pi \ii/q_n}:n\ge 0\}\rangle$ is an index $h = \lvert\ker(f)\rvert < \infty$ subgroup of $\hat{Z}$, and thus, for any continuous eigenvalue $\lambda$, we must have $\lambda^h \in \hat{\mathbb{Z}}_{(q_n)}$, i.e. any continuous eigenvalue is an $h$-th root of some eigenvalue in the dual of the known odometer, including the case where the new eigenvalue is just an $h$-th root of unity.
	
	Also, by this interpretation, we see that if $h > 1$ then $h$ is forced to be coprime to all but finitely many to the $q_n$'s: if $p$ was a common prime factor between $h$ and infinitely many of the $q_n$'s, then the group $\hat{\mathbb{Z}}_{(q_n)}$ would be closed under taking $p$-th roots, meaning that any continuous eigenvalue $\lambda\in\hat{Z}$ is an $(h/p)$-th root of some eigenvalue in $\hat{\mathbb{Z}}_{(q_n)}$, and in particular $\hat{\mathbb{Z}}_{(q_n)}$ could be at most an index $h/p$ subgroup of $\hat{Z}$, a contradiction.
	In particular, if the odometer $\mathbb{Z}_{(q_n)}$ is torsion-free, then $h$ will be coprime to all $q_n$ and thus $Z = \mathbb{Z}_{(q_n)}\times \mathbb{Z}/h\mathbb{Z}$.
	\end{remark}

\begin{example}\label{ex:extension_not_product}
		Consider the following two substitutions:
	\begin{align*} 
		\theta: a &\mapsto  abaca   & \tau: a    &\mapsto  ab
		\\  b  &\mapsto  babac & b  &\mapsto  bc\\
		c  &\mapsto  cabab
		 & c  &\mapsto ac,
	\end{align*}
	and the directive sequence $\boldsymbol{ \vartheta}=(\tau,\theta,\theta,\dotsc)$, 
	 so that $X_{\boldsymbol{ \vartheta}}=\tau(X_\theta)\cup\sigma\circ\tau(X_\theta)$. As $\theta$ is primitive and aperiodic, and $\tau$ is injective on letters, this decomposition ensures that $X_{\boldsymbol{ \vartheta}}$ is infinite and aperiodic. Note that although $\boldsymbol{ \vartheta}$ is not torsion-free, it is quasi-recognizable. To see this,  first, as $\theta$ is aperiodic and thus recognizable, there exists a factor map $\pi_{\rm tile}\colon X^{(1)}\to\Z_5$. Also, we only see the letter $a$ at the start of a $\tau$-tile, so there must exist a factor $\pi_2\colon X^{(0)}\to \Z/2\Z$, which is determined by the position of the instance of $a$ closest to the origin. Since $\tau$ is injective on letters, any $x\in X^{(0)}$ has a unique preimage $x^{(1)}\in X^{(1)}$, so we may define $\pi\colon X^{(0)}\to\Z_5\times \Z/2\Z$ by:
		\[\pi(x) = \bigl(\pi_{\rm tile}(x^{(1)}),\pi_2(x)\bigr).\]
	
	Thus, $\pi:X^{(0)}\rightarrow \Z_5\times \Z/2\Z$ is a factor map.
	Also, the substitution $\theta$ has height $2$, so $-1$ is an eigenvalue of $X_\theta = X^{(1)}$, with one associated eigenfunction $f\colon X^{(1)}\to S^1$ being given by:
	\[f(x) = \begin{cases}
		1 & \text{if }x_0 = a, \\
		-1 & \text{otherwise.}
	\end{cases}\]
	This $f$ induces an eigenfunction $\tilde{f}\colon X^{(0)}=X_{\boldsymbol{ \vartheta}}\to S^1$, which is given by:
	\[\tilde{f}(x) = \begin{cases}		f(y) & \text{if }x = \tau(y), \\
		e^{\pi \ii  /2}
\cdot f(y) & \text{if }x = \sigma(\tau(y)). 
	\end{cases}\]
	The injectivity of $\tau$ and the equality $\tau\circ\sigma=\sigma^2\circ\tau$ ensure that $\tilde{f}$ is well-defined and an eigenfunction for $X_{\boldsymbol{ \vartheta}}$, with eigenvalue  $e^{\pi \ii  /2}$. We can use this to verify that the maximal equicontinuous factor of $X_{\boldsymbol{ \vartheta}}$ equals $\Z_5\times\Z/4\Z$, which is a finite extension of the tiling factor $\Z_5\times\Z/2\Z$ by $\Z/2\Z$ (consistent with Theorem~\ref{thm: mef-c-to-one}) but is not isomorphic to $(\Z_5\times\Z/2\Z)\times\Z/2\Z$, so it is not a product of the odometer $\Z_{(q_n)}$ with the identified cyclic group. Note also that this is an example of a shift that is a $2$-to-$1$ extension of the (equicontinuous) odometer $(\Z_5 \times \Z/2\Z,+(1,1))$, but the latter does not equal the former's maximal equicontinuous factor.
\end{example}

\begin{corollary}\label{cor:MEF-torsion-free}
Let $\boldsymbol{\theta}$ be a  torsion-free directive sequence defined on a  sequence of bounded alphabets.
  Then  the maximal equicontinuous factor of  $(X_{\boldsymbol{\theta}},\sigma)$ is  $\Z_{(q_n)} \times \Z/h\Z$ for some $h$ coprime to  each $q_n$. In particular each continuous eigenvalue  for $(X_{\boldsymbol{\theta}}, \sigma)$ is rational.	
\end{corollary}
\begin{proof}
Let $\boldsymbol{\theta}$
have 
 length sequence $(q_n)$.
Since $\boldsymbol{\theta}$ is torsion-free, then by Theorem \ref{thm:dekking-sadic} we have an equicontinuous factor map $\pi_{\rm tile}\colon  X_{\boldsymbol{\theta}}\rightarrow \Z_{(q_n)}$ which maps limit words to $0$. The assumption that $q_n\geq 2$ infinitely often (by definition) tells us that $q_0 \cdots q_n \rightarrow \infty$, so by \cite[Lemma 5.13]{BSTY-2019}, there are finitely many limit words. Therefore the fibre $\pi_{\rm tile}^{-1}(0)$ is finite and we can apply Theorem \ref{thm: mef-c-to-one} to obtain the desired result.
 \end{proof}

\begin{remark}
We have stated Corollary \ref{cor:MEF-torsion-free} for torsion-free directive se\-quences, but we could also have stated it for a larger class of quasi-rec\-og\-niz\-a\-ble directive se\-quences. For example, given a quasi-rec\-og\-niz\-a\-ble constant-length directive sequence with length se\-quence $(q_n)$, if there exists $N$ such that $(q_n)_{n\geq N}$ is torsion-free, then the maximal equicontinuous factor of  $(X_{\boldsymbol{\theta}},\sigma)$ is  $\Z_{(q_n)_{n\geq N}} \times \Z/H\Z$ for some $H$ coprime to  each $q_n, n\ge N,$ where the $h$ of Theorem \ref{thm: mef-c-to-one} divides $H$.

Similarly, while  we give the following definition for torsion-free directive sequences, we can naturally extend it to the appropriate family of quasi-recognizable sequences.

\end{remark}

	\begin{definition}[Height]\label{def:dynamical-height}
	Let  $\boldsymbol{\theta}$ be a 
		 torsion-free directive sequence defined on a  sequence of bounded alphabets.
		 		 We call the $h=h({\boldsymbol \theta})$ guaranteed by Corollary \ref{cor:MEF-torsion-free}  the {\em height} of 
	 $\boldsymbol{\theta}$. If $h=1$, we say that $\boldsymbol{\theta}$ has {\em trivial} height.
	\end{definition}
	
	Note that this definition is consistent with \cite[Definition 6.7]{BCY-2022}. Note also that if $h$ is the height, then $\gamma(h)=h$, where $\gamma$ is defined before Lemma \ref{lem:partition}.
	 To avoid confusion, we remark that if  $\boldsymbol{\theta}$ is torsion-free, then its maximal equicontinuous factor is torsion-free {\em only} if $h=1$. 
	
	\subsection{Connection between height and coboundaries}
We connect our work to previous recent work \cite[Section 6]{BCY-2022}, where we associate to a con\-tin\-u\-ous eigenvalue a {\em coboundary}. As this commentary is simply to connect our work here to there, we do not include definitions, referring the reader to the aforementioned article for terminology. In that article, as the directive sequences were assumed finitary, telescoping was not always permitted, and there assumptions had to be made about the existence of a word of length 2 that belong to all languages $\mathcal L^{(n)}$.	As here we do not constrain the lengths of the morphisms to belong to a finite set, we may always telescope to obtain such words.  Also, there the results concerned {\em straight} directive sequences. But straightness can be obtained by telescoping, and  here also, as we can telescope arbitrarily, we can assume, without loss of generality, that our directive sequences are straight. Consequently, we can restate \cite[Theorem 6.6]{BCY-2022} as
\begin{theorem}\label{thm:fully-essential-constant-coboundarybis} 
Let $\boldsymbol{\theta} = (\theta^{(n)})_{n\ge0}$ be a torsion-free directive sequence defined on a  sequence of bounded alphabets, with length sequence $(q_n)_{n\geq 0}$.
   If  $\boldsymbol{\theta}$ has height $h$,
       then $\lambda$ is a continuous eigenvalue of $(X_{\boldsymbol \theta}, \sigma)$  if and only if
\begin{equation}\label{constant-coboundary-4} \lim_{n\rightarrow \infty}\lambda^{q_0 \cdots q_{n}}
\end{equation}
exists and is a constant coboundary which, if nontrivial, equals $\ee^{2 \pi \ii /\tilde{h}}$ with $\tilde{h}\mid h$.
 Furthermore $h$ divides $q_n -1$ for all $n$ large.\end{theorem}
 \begin{proof}
The proof that  the limit \eqref{constant-coboundary-4}  exists and defines a coboundary follows the same lines as the proof in \cite[Theorem 6.6]{BCY-2022}. By Theorem \ref{thm: mef-c-to-one}
we have that the limit in \eqref{constant-coboundary-4} must either equal 1, or $\ee^{2\pi \ii/\tilde{h}}$ where $\tilde h \mid h$. To see the last statement, from existence of \eqref{constant-coboundary-4} we conclude that 
\begin{equation*} \lim_{n\rightarrow \infty}\ee^{2\pi \ii {q_0 \cdots q_{n} (q_{n+1} -1)}/h}=1,\end{equation*}
and since $h$ is coprime to each $q_n$, the result follows.
 \end{proof}
From Theorem \ref{thm:fully-essential-constant-coboundarybis} we can extend \cite[Theorem 6.9]{BCY-2022}
 to obtain  a strength\-ened version of Cobham's theorem.  
 \begin{corollary}\label{cor:cobham}
Let $\boldsymbol{\theta} = (\theta^{(n)})_{n\ge0}$  and $\boldsymbol{\tau} = (\tau^{(n)})_{n\ge0}$ be two  torsion-free  constant-length directive sequences,
defined on a sequence of bounded alphabets,
 with length sequences $(q_n)_{n\geq 0}$ and
$(\tilde{q}_n)_{n\geq 0}$. 
If there is a prime factor of  some $q\in  \{q_n:n\geq 0\}$ that  is not a prime factor of any $\tilde{q}\in \{\tilde{q}_n:n\geq 0 \}$, then $(X_{\boldsymbol{ \theta}},\sigma)$ cannot be a topological factor of $(X_{\boldsymbol{ \tau}},\sigma)$.
\end{corollary}

\begin{example}
	Let $\boldsymbol{ \theta} = (\theta^{(j)})_{j\ge 0}$ be any directive sequence with lengths given by $q_j=(j+2)!$ and such that $ X_{\boldsymbol{\theta}}^{(0)}  $ is aperiodic; for instance, one may take the directive sequence given by:
\begin{align*}
	\theta^{(j)}\colon 0 &\mapsto 0^{(j+2)!-1}1 \\
	1 &\mapsto 1^{(j+2)!-1}0,
\end{align*}
where $q_j=(j+2)!$. It is easy to see that this is a torsion-free directive sequence, as $p\mid q_j$ implies $p\mid q_{j'}$ for any $j'>j$.

The odometer $\Z_{(q_n)}$ given by the length sequence $(q_n)$ is isomorphic to the product $\Omega\coloneqq \prod_{p\text{ prime}}\Z_p$ of every $p$-adic odometer, for prime $p$; this is often called the \emph{universal odometer} \cite{Downarowicz2005}, as any odometer is a factor of $\Omega$. The maximal equicontinuous factor of $(X_{\boldsymbol{ \theta}}, \sigma)$  is then necessarily at most a finite extension of $\Omega$ 
 by a group of order $h$, the dynamical height.

As observed in Remark~\ref{rem:pontryagin_interpretation}, $h$ must be coprime to all numbers $(n+2)!$, which is only possible if $h=1$. Thus, $\Omega$ is  the maximal equicontinuous factor of $X_{\boldsymbol{ \theta}}$. An alternative interpretation is that every rational in $[0,1)$ is already an additive eigenvalue of the system, so we cannot add any new eigenvalue.

\end{example}

\begin{example}\label{ex:Durand}
		Consider  the substitutions  $S=\{\theta,  \tau\}$ with
	\begin{align*}
		\theta\colon a   &\mapsto acb    &\tau\colon  a &   \mapsto  abc
		\\  b  &\mapsto  bab  & b&  \mapsto acb\\
		 c  &\mapsto  cbc
		& c  & \mapsto aac.
	\end{align*}
	These two substitutions were defined by Durand in \cite{Durand-2003}, where he took a specific directive sequence on $\{ \theta, \tau\}$ to produce a finitary strongly primitive constant-length directive sequence whose associated shift  is minimal, but not linearly recurrent. Here we consider any  directive sequence taking values from 	$\{ \theta, \tau\}$. As shown in \cite[Example 1]{BCY-2022}, any directive sequence is rec\-og\-niz\-a\-ble and hence $ X_{\boldsymbol{\theta}}^{(0)}  $ is aperiodic; (alternatively, it is injective and torsion free, so recognizable). It can also be
	 verified that it is primitive.
	 Thus by Corollary \ref{cor:MEF-torsion-free} we conclude that for each directive sequence the maximal equicontinuous factor of the corresponding shift is $(\Z_{(q_n)} \times \Z/h\Z, (+1,+1))$ where $h$ may depend on ${\boldsymbol \theta}$. We argue that $h$ always equals 1. One can show that no matter the selected directive sequence, one can telescope so that there is a word $\alpha\alpha\in \mathcal L^{(n)}$ for each $n$. If $\ee^{2 \pi \ii /h}$ is an eigenvalue, then the existence of this word allows us to conclude that  $\ee^{2\pi \ii p_n/h}\rightarrow 1$. Since $h$ is coprime to each $p_n$, this forces $h=1$.

\end{example}

\begin{example}\label{ex:nontrivial height}
This is a modification of \cite[Example 6.10]{BCY-2022}.
Take any finite set $\mathcal Q$ of odd numbers, and let $\mathcal S$ be any set of constant-length substitutions on $\mathcal A=\{a,b,c,d \}$ where
\begin{enumerate}
\item   for  each  $\alpha \in \{a,b\}$ and each $\theta\in \mathcal S$,  $\theta(\alpha)$ starts  with a letter in $\{a,b\}$ and similarly  for  each  $\alpha \in \{c,d\}$,  $\theta(\alpha)$ starts  with a letter in $\{c,d\}$,
 \item
for each $\theta\in \mathcal S$, any occurrence of  a letter in $\{a, b\}$ in  the image of letter by a  substitution is always followed by a letter in $\{c, d\}$, and any occurrence of a letter in $\{c, d\}$ is always followed by a letter  $\{a, b\}$, and
\item each substitution in $\mathcal S$ has length belonging to $\mathcal Q$.
\end{enumerate}
Then we claim that $-1$ is a continuous eigenvalue for any primitive directive sequence where any substitution in $\mathcal S$ that appears in $\boldsymbol \theta$ appears infinitely often, and where $ X_{\boldsymbol{\theta}}^{(0)}  $ is aperiodic. The previous assumptions imply  that $\boldsymbol \theta$ is torsion-free, so that Corollary \ref{cor:MEF-torsion-free} applies.
Also, the three conditions above ensure that $-1$ is an eigenvalue, as
  \[ \mathcal P = \{ [a]\cup[b], [c]\cup [d]\}\]
  is then a clopen partition which forms a Rokhlin tower of height 2, so that $2\mid h$. Note that Corollary \ref{cor:alphabet_partition} tells us that this kind of example is essentially the only way that height can manifest.
  \end{example}

		\subsection{The pure base of a torsion-free directive sequence}\label{subsec:pure-base}
		In this section we assume that  the constant-length directive sequence is defined 
		 on a se\-quence of bounded alphabets, but for ease of notation we will give proofs for the case when $\boldsymbol{\theta}$ is defined
		on  $\mathcal A$. We assume that 
		 it is torsion-free, so that it is quasi-recognizable
				 by Theorem \ref{thm:dekking-sadic}.
		  If $\boldsymbol{\theta}$  has height $h>1$ we would like to define a {\em pure base} in a manner analogous to that defined by Dekking for constant-length substitutions; 
		let us recall how to define the pure base of a length-$\ell$ substitution.
				Given a primitive length-$\ell$ substitution $\theta\colon \mathcal A\rightarrow \mathcal A^{\ell}$,  a fixed point $u$, and some $h\in \N$,   consider the set of words
		\[  \mathcal W:=\{ u_{kh} \cdots u_{(k+1)h -1 }: k\in \N\};\] 
		it is finite with cardinality $c$. Define an alphabet $\mathcal B:= \{b_1,\dots b_c\}$ such that  each letter in $\mathcal B$  codes a distinct word in the above set.
				Let $\bar \tau \colon \mathcal B\rightarrow \mathcal W$ be the natural map which assigns to a letter in $\mathcal B$ the word in $\mathcal W$ which it codes. Conversely, 
				let $\tau: \mathcal W\rightarrow \mathcal B$ be the inverse of $\bar \tau$.
		The map $\tau$ extends by concatenation to words on $\mathcal W$. 		Define
		 $\bar \theta \colon \mathcal B\rightarrow \mathcal B^{\ell}$ as 
		\begin{equation}\label{def:pure-base-subs}  \bar\theta(b)  = \tau(\theta(\bar \tau(b)));\end{equation}
		$\bar\theta$ gives  what is known as a {\em $h$-th higher power shift presentation  of  $X_{\theta}$}, and it is called 		
		 a {\em pure base} of $\theta$.
		 In fact, as shown in \cite[Remark 9, Lemmas 17 and 19]{Dekking1977},
$(X_\theta,\sigma)$ is conjugate to
a constant height suspension over $(X_{\bar\theta}, \sigma)$. If the height of the stationary directive sequence $(\theta, \theta,\dotsc)$ is $h$, then this suspension has height $h$.

We continue with the notation above, namely the maps $\tau$ and $\bar\tau$.

\begin{definition}[Pure base] Let 
$\boldsymbol{\theta}$ be a constant-length $(q_n)$ directive sequence on $\mathcal A$. Let $(u^{(k)})_{k\geq 0}$ be such that $\theta^{(k)}(u^{(k+1)})=u^{(k)}$ for each
	$k$, and
	 let $\mathcal W^{(n)}$ be the set of words of length $h$ that appears at the  indices $\{kh: k\geq 0\}$ in $u^{(n)}$. Code  $\mathcal W^{(n)}$ with an alphabet
	 $\mathcal B^{(n)} $ of cardinality $|\mathcal W^{(n)}|$.
	  Let $\bar \tau_n\colon \mathcal B^{(n)}\rightarrow \mathcal A^{h}$ be the natural map which associates to a letter in $\mathcal B^{(n)}$ its representative
	   in $\mathcal W^{(n)}$. Conversely, if $w\in \mathcal W^{(n)}$,  let $\tau_n:\mathcal W^{(n)}\rightarrow \mathcal B^{(n)}$ be the inverse of $\bar\tau_n$.  The map $\tau_n$ extends to concatenations of words over $\mathcal W^{(n)}$, and similarly the map $\bar \tau_n$ extends to concatenations of letters.
		Define
		 ${\bar \theta}^{(n)} \colon \mathcal B^{(n+1)}\rightarrow (\mathcal B^{(n)})^{q_n}$ as 
		\begin{equation}\label{def:pure-base}  {\bar\theta}^{(n)}(b)  = \tau_n(\theta^{(n)}({\bar \tau}_{n+1}(b))).\end{equation}
		Then we call the directive sequence ${\boldsymbol{\bar\theta}}=( {\bar \theta}^{(n)})_{n\geq 0}$  the {\em pure base } of ${\boldsymbol \theta}$.  \end{definition}
		We shall see in Theorem \ref{thm:height} that if ${\boldsymbol \theta}$ has height $h>1$,  this construction will
		give  a $h$-th higher power shift of    $(X_{\boldsymbol \theta},\sigma)$.

\begin{lemma}\label{lem:primitive}
If ${\boldsymbol \theta}$ is a  constant-length directive sequence on $\mathcal A$ such that each $(X_{\boldsymbol {\theta}}^{(n)},\sigma)$ is minimal and  $ X_{\boldsymbol{\theta}}^{(0)}  $ is aperiodic,
then each $(X_{\boldsymbol {\bar\theta}}^{(n)},\bar\sigma)$ is min\-i\-mal.
	\end{lemma}	
	
	\begin{proof} We show that the one-sided shift $(\tilde{X}_{\boldsymbol {\bar\theta}}^{(n)},\tilde{\sigma})$ is minimal; this implies that  the two-sided shift   $(X_{\boldsymbol {\bar\theta}}^{(n)},\bar\sigma )$  is minimal.
		Suppose that the limit word sequence $(u^{(n)})_{n\geq 0}$ is used to define  ${\boldsymbol {\bar\theta}}$. To show that 
		each $(\tilde{X}_{\boldsymbol {\bar\theta}}^{(n)},\tilde{\sigma})$ is minimal
			we first construct a  sequence $(\bar {u}^{(n)})_{n\geq 0}$ such that ${\bar \tau}_n( \bar {u}^{(n)} ) =  u^{(n)}$. Define $\bar u = {\bar u}^{(0)}$ as the unique sequence such that 
${\bar \tau}_0({\bar u}^{(0)})=u$.
Next, given $n\geq 0$ and knowledge of $u_0^{(n)}\cdots u_{q_n(k+1)h-1}^{(n)}$, define the unique word $ \bar{u}_0^{(n+1)} \cdots \bar{u}_{k}^{(n+1)}\in \mathcal L^{(n+1)}_{  \boldsymbol {\bar\theta}}$ which satisfies 
\[\theta^{(n)} {\bar \tau}_{n+1} ( \bar{u}_0^{(n+1)} \cdots \bar{u}_{k}^{(n+1)}) = u_0^{(n)}\cdots u_{q_{n}(k+1)h-1}^{(n)}. \]
Then for each $n$, the sequence of nested words $(\bar{u}_0^{(n)} \cdots \bar{u}_{k}^{(n)})_{k\geq 0}$ converges to  a sequence $\bar u^{(n)}$ and
\begin{align*} {\bar \theta}^{(n)} (        \bar u^{(n+1)}  )&=  {\bar \theta}^{(n)} (  \lim_{k\rightarrow \infty}      {\bar u}^{(n+1)}_0 \cdots   {\bar u}^{(n+1)}_k  ) = 
 \tau_n {\theta}^{(n)}    \bar \tau_{n+1}(  \lim_{k\rightarrow \infty}      {\bar u}^{(n+1)}_0 \cdots   {\bar u}^{(n+1)}_k  ) \\& =   \tau_n    \left(  \lim_{k\rightarrow \infty}u_0^{(n)}\cdots u_{q_{n}(k+1)h-1}^{(n)} \right)  = \bar u^{(n)},               \end{align*}
 where the last step follows by an inductive argument. See Figure~\ref{fig:inductive-pure-base}. Now minimality implies that for  each $n$, if  $u^{(n)}$ contains a word, then this word appears uniformly recurrently in $u^{(n)}$. Furthermore, as the cyclic $\sigma^h$-minimal partition consists of clopen sets, then if a word is long enough, it only appears at indices that are congruent to a fixed $i\bmod h$. Therefore, if $w$ is a sufficiently long word which occurs at an index congruent to a fixed $i\bmod h$, then it  occurs
 uniformly recurrently at an index congruent to a fixed $i\bmod h$, and its image under  $\tau_n$ occurs uniformly recurrently in ${\bar u}^{(n)}$. Thus 
  any word that appears in ${\bar u}^{(n)}$ must also appear uniformly recurrently. The result follows.
	\end{proof}
\begin{figure} 
	\includegraphics{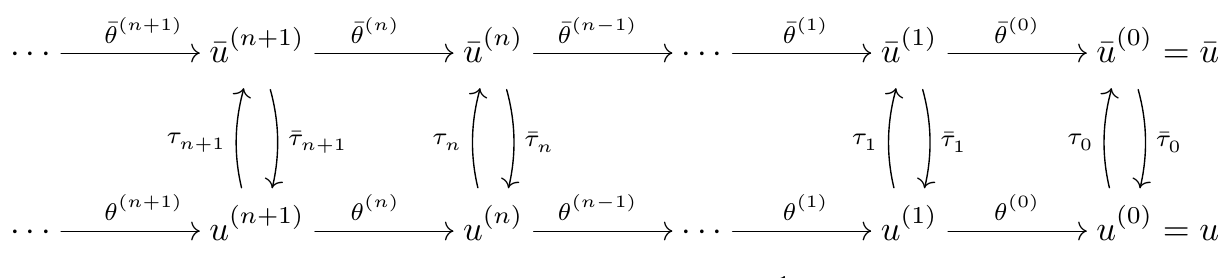}
	\caption{Using the maps $(\tau_n)$ and $(\bar \tau_n)$ to transfer a ${\boldsymbol \theta}$-limit word to a ${\boldsymbol {\bar \theta}}$-limit word}
	\label{fig:inductive-pure-base}
\end{figure}
	
	\begin{theorem}\label{thm:height}
If ${\boldsymbol \theta}$  is a torsion free directive sequence defined on 
$\mathcal A$, of height  $h$, then
 the pure base $ \boldsymbol{ \bar\theta}$ is torsion-free and  has trivial height, and  
$(X_{\boldsymbol\theta}, \sigma) 
\cong(X_{ \boldsymbol{ \bar  \theta}}\times \{0,\ldots,h-1\},T)$ where

\begin{equation*}
 T(x,i)\coloneqq 
\begin{cases}
  (x,i+1)      & \text{ if  } 
0\leq  i<h-1    \\
    (\bar \sigma(x), 0)   & \text{ if } i=h-1
\end{cases}
\end{equation*} 

\end{theorem}

\begin{proof}
Fix the sequence $(u^{(n)})_{n\geq 0}$, such that   $\theta^{(n)}(u^{(n+1)}) = u^{(n)}$ and $u=u^{(0)}\in X_{\boldsymbol \theta}$,
 that defines the directive sequence $\boldsymbol{\bar \theta}$.
By definition, the length sequence of ${\boldsymbol{\bar \theta}}$ is also $(q_n)$,   the length sequence of ${\boldsymbol{ \theta}}$. 
By Lemma \ref{lem:primitive}, each shift $(X_{\boldsymbol{\bar \theta}}^{(n)},\bar\sigma)$ is minimal, and also it cannot be periodic, as $X_{\boldsymbol \theta}$ is not periodic.
Hence $\boldsymbol{\bar \theta}$ is torsion free.
Note that  ${\bar \theta}^{(n)}$ is defined on $(\mathcal B_n)$ where each $\lvert\mathcal B_n\rvert\leq \lvert\mathcal A\rvert^{h}$. 
Thus
by Theorem  \ref{thm:dekking-sadic}, each of  ${\boldsymbol{ \theta}}$ and  ${\boldsymbol{\bar \theta}}$ is quasi-recognizable.
 As $\gamma_{\boldsymbol{\theta}}(h)=h$, there exists a $\sigma^h$-minimal set $X_0\subset X_{\boldsymbol {\theta}}$ containing $u$. We claim that the map 
${\bar \tau}_0 \colon X_{\boldsymbol {\bar\theta}} \rightarrow X_{\boldsymbol \theta}$  is a bijection between $X_{\boldsymbol {\bar\theta}}$ and $X_0$.  
Let ${\bar u}$ be the unique sequence such that 
${\bar \tau}_0({\bar u})=u$, i.e., $\tau_0(u)  ={\bar u}.   $ 
 As $X_{\boldsymbol{\bar \theta}}$ is minimal  by Lemma \ref{lem:primitive},  for any $x\in X_{\boldsymbol{\bar \theta}}$ we can write, for some $(n_k)$, 
 \begin{align*}
{\bar \tau }_0(x)& 
=  {\bar \tau}_0 \left(\lim_k \bar{\sigma}^{n_k} \tau_0 (  {u}   ) \right) = 
 {\bar \tau}_0 \left(\lim_k \tau_0 \sigma^{n_k h}  (  {u}   ) \right) =    \lim_k  \sigma^{n_k h}  (  {u}   ),
\end{align*}
and since  ${u}\in X_0$ and $X_0$ is $\sigma^h$-invariant, therefore ${\bar \tau }_0(x)\in X_0$. Also,
\begin{align*}{\bar \tau}_0  \bar\sigma (x) &= {\bar \tau}_0 \bar\sigma (     \lim_k {\bar\sigma}^{n_k} (  {\bar u} )     ) =
   {\bar \tau}_0 \bar \sigma (     \lim_k \bar\sigma^{n_k} \tau_0 (  { u}  )     ) =  {\bar \tau}_0  (   \tau_0  \lim_k \sigma^{(n_k+1) h}  (  { u} )     )\\&=
       \sigma^h  \lim_k \sigma^{n_k h}  (  { u}  )      =   \sigma^h  {\bar \tau}_0 \left( \lim_k {\bar\sigma}^{n_k }  (  {\bar u }  )\right)     =
 \sigma^h {\bar \tau}_0 (x) \end{align*}
 
 so that $\bar \tau_0$ is a conjugacy between
  $(X_{\boldsymbol{\bar \theta}}, \bar\sigma)$ and $(X_0, \sigma^h)$, and the second statement follows. 
 
 Next we show that $\boldsymbol{\bar \theta}$ has trivial height. Suppose $\bar h$ is coprime to $(q_n)$ and  that $Y_0\subset X_{\boldsymbol{\bar \theta}}$ is $\sigma^{\bar h}$-minimal and 
 defines a  ${\bar\sigma}^{\bar h}$-cyclic  partition    
 $\{  {\bar\sigma}^{i }Y_0 : 0\leq i \leq \bar h -1\}$
 of  $X_{\boldsymbol{\bar \theta}}$. Then ${\bar \tau}_0 (Y_0)$       is $\sigma^{h\bar h}$-minimal  in $X_{\boldsymbol  \theta}$ and 
$ \{  \sigma^{i} {\bar \tau}_0 (Y_0), 0\leq i \leq { h\bar h} -1\}$ is a $h\bar h$-cyclic partition of $X_{\boldsymbol \theta}$. But the definition
 of $h$ as being maximal forces $\bar h=1$.
\end{proof}

\section{Combinatorial interpretations of height}\label{sec:combinatorial-height}

In Definition \ref{def:dynamical-height}, we defined the height  of a directive sequence as the maximal $h$ such that $\gamma(h)=h$ and $h$ is co-prime 	to the length sequence $(q_n)_{n\geq 0}$. In this section we give a combinatorial characterisation of height in terms of return times, as was done in \cite{Dekking1977} for substitutions.

Recall that  for a primitive substitution of length $\ell$ with fixed point $u=u_0,u_1 \dots$, an equivalent definition of the height is
\[
h(\theta)\coloneqq  \max \{n\geq 1: \gcd(n,\ell)=1, n \mid \gcd\{k: u_k=u_0 \} \}\, .
\]

In  this section we find an equivalent combinatorial definition of the height for the directive sequence $\boldsymbol \theta$. 
By Proposition~\ref{prop:two-one-sided-continuous},  since the height gives rise to a continuous eigenvalue, we can work in the one-sided setting.

\begin{remark} For some technical reasons, e.g., see Lemma \ref{lem:comb-height-eigenvalue} and Example \ref{ex:injectivisation-2} below, we will sometimes need to work with an injective directive sequence. Theorem \ref{thm:sadic-recognizable}
 tells us  that we may always replace a directive sequence $\boldsymbol{\theta}$,  by a recognizable $\boldsymbol{ \widehat  \theta}$,   such that $X_{\boldsymbol{ \theta}} = X_{\boldsymbol{ \widehat  \theta}}$, and where every morphism in $\boldsymbol{ \widehat  \theta}$ is injective on letters.
	In light of this, we will define a sequence of combinatorial heights  $(h^{(n)})$  using the level-$n$ shifts of the injectivisation $\boldsymbol{\widehat \theta}$, and  for the remainder of this section we work with directive sequences that are injective on letters.
\end{remark}

Let $\boldsymbol{ \theta}$ be an injective directive sequence. Let $u^{(0)}$ be a limit word for $\boldsymbol{ \theta}$, so that there is a sequence $(u^{(n)})_{n\geq 0}$ satisfying $u^{(n)}\in X^{(n)}$ and 
$\theta^{(n)}(u^{(n+1)})= u^{(n)}$, i.e., $\theta^{[0,n)}(u^{(n)})=u^{(0)}$ for each $n$. We write $\gcd(m,(q_n))=1$ if $m$ is coprime to $q_n$ for each $n$.
Define 
\[h^{(n)}(\boldsymbol{ \theta} ) =h^{(n)}\coloneqq   \max \{d\geq 1: \gcd(d,(q_N)_{N\geq n})=1, d \mid \gcd\{k: u_k^{(n)}=u_0^{(n)} \} \}. \]
If $\boldsymbol{\theta}$ is not injective, we define 
\[h^{(n)}(\boldsymbol{ \theta})\coloneqq h^{(n)}    (\boldsymbol{ \widehat\theta})      \]
where $\boldsymbol{ \widehat\theta}$ is the injectivisation of $\boldsymbol{ \theta}$ given by 
by Theorem \ref{thm:sadic-recognizable}.

\begin{example} 
Consider  the substitutions  $S=\{\theta, \tau\}$ with
	\begin{align*}
		\theta\colon a   &\mapsto  aba  \ \  &\tau\colon a  &  \mapsto aab
		\\  b&  \mapsto  bac \ \  & b&  \mapsto  abc\\
		\ c  &\mapsto  bab
		\ \ & c &  \mapsto aac
	\end{align*}
	and consider the directive sequence $(\tau, \theta,\theta,\theta,\dotsc)$; then $h^{(n)}=2$ for $n\geq 1$, but $h^{(0)}=1$.
\end{example}

The previous example tells us that the sequence of combinatorial heights $(h^{(n)})$ can fluctuate. Nevertheless, if some $h^{(n)}>1$, injectivity forces the existence of eigenvalues for $(X^{(m)}, \sigma)$ for $m\leq n$.

\begin{lemma}\label{lem:comb-height-eigenvalue}
Let $\boldsymbol{ \theta}$ be an injective  torsion-free directive sequence defined on a  sequence of bounded alphabets. 
If $h^{(n)}\geq 2$, then 
 $(X^{(m)},\sigma)$ has a cyclic $h^{(n)}$-minimal partition for $0\leq m \leq n$. 
\end{lemma}
\begin{proof}
Suppose that $h\coloneqq h^{(n)}\geq 2$. 
We first show that $(X^{(n)},\sigma)$ has a  cyclic $\sigma^{h}$-minimal  partition. Since $(X^{(n)},\sigma)$ is minimal, there is an $\ell$ such that any word of length $\ell$ contains at least one occurrence of the letter $u^{(n)}_0$.  Let $Q_i$ be the set of such words $w$ where we see the first occurrence of $u^{(n)}_0$ at a location congruent to  $i \bmod h$. Taking $C_i\coloneqq  \bigcup_{w\in Q_i} [w]$, we obtain that $\{C_0, \dotsc, C_{h-1} \}$ is a   cyclic $h$-minimal partition; the distance between two instances of the letter $u_0^{(n)}$ is a multiple of $h^{(n)}$ by definition (and minimality), ensuring that the definition of this partition is consistent.

Next, we claim  that  for any $m\leq n$, we can create a cyclic $h$-minimal partition for $(X^{(m)},\sigma)$. To do this we will proceed inductively, showing that we can build a cyclic $h$-minimal partition for $X^{(m-1)}$ from one such partition in $X^{(m)}$ whenever $\theta^{(m-1)}$ is injective on letters.

Let $\{C_0^{(m)},\dotsc,C_{h-1}^{(m)}\}$ be a cyclic $\sigma^h$-minimal partition for $X^{(m)}$. Consider the following col\-lec\-tion of sets:
 	\[C_{j,k}^{(m-1)}\coloneqq \sigma^j(\theta^{(m-1)}(C^{(m)}_k)),\quad 0\le j < q_{m-1},0\le k < h.\]
By definition, $\sigma(C_{j,k}^{(m-1)})=C_{j+1,k}^{(m-1)}$ whenever $0\le j < q_{m-1}-1$; note that the identity $\sigma^{q_{m-1}}\circ\theta^{(m-1)}=\theta^{(m-1)}\circ\sigma$ ensures that $\sigma(C_{q_{m-1}-1,k}^{(m-1)})=C_{0,k+1}^{(m-1)}$, where the second index is taken modulo $h$. Thus, if we define
	\[C_k^{(m-1)} \coloneqq \bigcup_{\substack{0\le \ell < h \\ j\cdot q_{m-1} + \ell\equiv k\pmod{h}}} C_{j,\ell}^{(m-1), }\]
then we must have that $\sigma(C_k^{(m-1)})=C_{k+1 \pmod{h}}^{(m-1)}$. Since $\theta^{(m-1)}$ is injective on letters, and since ${\boldsymbol \theta}$ is torsion-free and hence recognizable, $\theta^{(m-1)}$
 is injective as a function $X^{(m)}\to \theta^{(m-1)}(X^{(m-1)})$. Thus, the $hq_{m-1}$ sets $C^{(m-1)}_{j,k}$ are all disjoint, ensuring that the sets $\{C^{(m-1)}_0, \dots C^{(m-1)}_{h-1}\}$ form a partition. We conclude by induction.
\end{proof}

We now study how the eigenfunctions at different levels $X^{(n)}$ are related, and the combinatorial interpretation of this relationship. We start with the following simple observation. To avoid confusion with what will follow we temporarily call the height of  Definition~\ref{def:dynamical-height} the {\em dynamical} height.

\begin{lemma}\label{lem:eigenfunctions_at_higher_levels}
	Let $\boldsymbol{ \theta}$ be a torsion-free directive sequence with dynamical height $h$. Then for each $n\ge 1$,   $\lambda = \ee^{2\pi\ii/h}$  is 
		 an eigenvalue for  $(X^{(n)}, \sigma)$.
\end{lemma}

\begin{proof}
	Let $f_0\colon X^{(0)}\to S^1$ be an eigenfunction associated to  $\lambda$, and define  $f_1\colon X^{(1)} \to S^1$ given by $f_1(x) \coloneqq (f_0\circ\theta^{(0)}(x))^{r_0}$, where $r_0$ is an inverse modulo $h$ of $q_0$, i.e.,  $q_0r_0 \equiv 1 \pmod{h}$.
 Then 	 we have
		\begin{align*}
			f_1\circ\sigma(x) &= (f_0 \circ \theta^{(0)} \circ \sigma(x))^{r_0} \\
			&= (f_0\circ \sigma^{q_0} \circ \theta^{(0)}(x))^{r_0} \\
			&= (\lambda^{q_0}\cdot f_0 \circ \theta^{(0)} (x))^{r_0} \\
			&= \lambda^{q_0 r_0}\cdot (f_0 \circ \theta^{(0)} (x))^{r_0} \\
			&= \lambda \cdot f_1(x),
		\end{align*}
		Thus $f_1$ is a continuous eigenfunction for $X^{(1)}$ with associated eigenvalue $\lambda$. Inductively, given an eigenfunction $f_n\colon X^{(n)}\to S^1$, we define an eigenfunction $f_{n+1}\colon X^{(n+1)} \to S^1$ by $f_{n+1}(x) \coloneqq (f_n\circ \theta^{(n)}(x))^{r_n}$, where $q_n r_n \equiv 1 \pmod{h}$; such $r_n$ always exists as $h$ is coprime to every $q_n$.
\end{proof}

\begin{corollary}\label{cor:alphabet_partition}
	Let $\boldsymbol{ \theta}$ be a torsion-free directive sequence defined on a sequence of bounded  alphabets $(\mathcal A_n)$, and with dynamical height $h$. Then there exists some $n^*\in\N$ such that for every $n\ge n^*$, there is a partition $\{\mathcal{A}_0^{(n)},\dotsc,\mathcal{A}^{(n)}_{h-1}\}$ of $\mathcal{A}_n$ into $h$ sets, 
	 such that for any $ab \in \mathcal L^{(n)}$, if $a \in\mathcal{A}_j^{(n)}$, then $b\in\mathcal{A}_{j+1 \pmod h}^{(n)}$. 
	 \end{corollary}
\begin{proof}
	We work with the one-sided shifts.	 Let $(f_n)$ be the eigenfunctions in the proof of Lemma~\ref{lem:eigenfunctions_at_higher_levels}, associated to the eigenvalue $e^{2\pi i /h}$. Since $f_0$ is continuous, there exists a value $N_0\ge 1$ such that $x\rvert_{[0,N_0)}=y\rvert_{[0,N_0)}\implies f_0(x)=f_0(y)$. Since $\lvert\theta^{(0)}(w)\rvert = q_0 \cdot \lvert w \rvert$ for any word $w$, we have that if $N_1 = \bigl\lceil \frac{N_0}{q_0}\bigr\rceil$, then for each $x^{(1)},y^{(1)}\in X^{(1)}$:
		\begin{align*}
			x^{(1)}\rvert_{[0,N_1)}=y^{(1)}\rvert_{[0,N_1)} & \implies \theta^{(0)}(x^{(1)}) \rvert_{[0,N_0)}=\theta^{(0)}(y^{(1)})\rvert_{[0,N_0)}\\
			 &\implies f_1(x^{(1)}) = f_0(\theta^{(0)}(x^{(1)}))=f_0(\theta^{(0)}(y^{(1)})) = f_1(y^{(1)}),
		 \end{align*}
	 Similarly, if we define inductively $N_{j+1}=\bigl\lceil\frac{N_j}{q_j}\bigr\rceil$, we see that $f_j(x^{(j)})$ is entirely determined by $x^{(j)}\rvert_{[0,N_j)}$. Since infinitely many of the $q_j$ are greater than $1$, 
	  there exists $n^*$  such that if $n\ge n^*$ then  $N_{n} = 1$, so that  $f_n$ is determined entirely by the letter at the origin,
	i.e., there exists a function $\bar{f}_n\colon\mathcal{A}\to S^1$ such that $f_n(x^{(n)})=\bar{f}_n(x^{(n)}_0)$.
	 Defining $\mathcal{A}^{(n)}_j=\bar{f}_n^{-1}(\{\lambda^j\}), 0\le j< h$, we obtain a partition of $\mathcal{A}$ into $h$ non-empty sets, where the series of equalities 
	 \[\bar{f}_n(x_1^{(n)})=f_n\circ\sigma(x^{(n)})=\lambda\cdot f_n(x^{(n)}) = \lambda\cdot\bar{f}_n(x^{(n)}_0)\]
	 imply that for any $x^{(n)}\in X^{(n)}$, if $x_0^{(n)}\in\mathcal{A}^{(n)}_j$ then $x_1^{(n)}\in\mathcal{A}^{(n)}_{j+1}$, as desired.
\end{proof}

\begin{corollary}\label{cor:heights-stabilise}
Let $\boldsymbol{\theta}$ be a  torsion-free directive sequence which is injective on letters and defined on a sequence of bounded alphabets. Then the sequence $(h^{(n)})$ is bounded.
\end{corollary}

\begin{proof}
If the sequence $(h^{(n)})$ is not bounded, then by Lemma $\ref{lem:comb-height-eigenvalue}$ we  obtain  that  $(X^{(0)}, \sigma)$ has a cyclic $\sigma^{j_n}$-minimal partition for arbitrarily large  $j_n$, equivalently, that $\ee^{2\pi\ii/j_n}$ is a continuous eigenvalue of the shift. But this contradicts Corollary \ref{cor:MEF-torsion-free}, since any $j_n$ coprime to $(q_j)_{j\ge 0}$ must be a divisor of the dynamical height $h$. Therefore the sequence $(h^{(n)})$ is bounded. \end{proof}
Thus we can define the \emph{combinatorial height} for $\boldsymbol \theta$ as
\begin{equation}\label{eq:height-combin}
 h_{\rm comb}(\boldsymbol \theta)\coloneqq\max \{ h^{(n)}: n\geq 0 \} 
\end{equation}

We will show, in Theorem \ref{thm:combinatorial-height}, that the combinatorial height given by \eqref{eq:height-combin} equals the (dynamical) height  given in Definition \ref{def:dynamical-height}.
Before we do that we give an example to show  why we need to define the sequence $h^{(n)}$ in terms of the injectivisation $\boldsymbol{\widehat\theta}$ of $\boldsymbol{\theta}$.

\begin{example}\label{ex:injectivisation-2}
		To see how combinatorial height may fail to reflect the actual height of the shift in the absence of injectivity, consider the three morphisms $\theta,\vartheta,\varrho$ from Example~\ref{ex:quasi_recognizable_but_not_recognizable} and the same directive sequence $\boldsymbol{\alpha} = (\varrho,\vartheta,\vartheta,\vartheta,\dotsc)$. 
	
	As noted in the previous example,  we have $X^{(0)} = X_\theta$, and $ X^{(n)} = X_\vartheta$ for any $n\ge 1$. We can easily verify that $X_\vartheta$ has (dynamical) height $2$, as its alphabet partitions into $\{\{0,1\},\{\bar{0},\bar{1}\}\}$, where a barred symbol is always followed by an unbarred symbol, and vice versa. Hence, the directive sequence $\boldsymbol{ \alpha}$ has an alphabet partition into two sets at every level from $1$ onwards.
	
	If we ignore the the injectivity hypothesis for a moment, we could compute $h^{(n)}$ by definition for the sequence $\boldsymbol{\alpha}$, obtaining $h^{(0)}=1,h^{(n)}=2$ for $n\ge 1$ due to the presence of the aforementioned alphabet partition. We could be tempted to conclude that $X_{\boldsymbol{\alpha}}$ has dynamical height at least $2$; however, as $11$ is in the language of the substitutive shift $X_\theta = X_{\boldsymbol{ \alpha}}$, the latter must be pure (i.e. the odometer $\Z_3$ is already its maximal equicontinuous factor), and thus have height $1$.
\end{example}

The following result is a generalisation of \cite[Lem.~11(ii)]{Dekking1977}.
\begin{lemma}\label{lem:EV-div}
Let $\boldsymbol{\theta}$ be a  torsion-free directive sequence defined on a sequence of bounded al\-pha\-bets $(\mathcal{A}_n)$, with length-sequence $(q_n)$.
Let $\ee^{2\pi \ii /m}$ be a con\-tin\-u\-ous eigenvalue for $(X_{\boldsymbol\theta}, \sigma)$ with $(m,q_n)=1$ for each $n$.  Then $m\mid h^{(n)}$ for all $n$ large.
\end{lemma}		 
\begin{proof}   Theorem \ref{thm:sadic-recognizable} tells us that if   $\boldsymbol\theta$ is not injective, we can equally work with its injectivisation   $\boldsymbol{\widehat\theta}$. Henceforth we assume that $\boldsymbol\theta$ is injective. 
We know from Corollary~\ref{cor:alphabet_partition} that for any sufficiently large $n$ we can find a partition of the alphabet into $h$ disjoint sets $\{\mathcal{A}_0^{(n)}, \dotsc , \mathcal{A}_{h-1}^{(n)}\}$  such that in every point of $X^{(n)}$, whenever we see a symbol from $\mathcal{A}_j^{(n)}$, it is followed by a symbol from $\mathcal{A}_{j+1}^{(n)}$. In particular, the next symbol from $\mathcal{A}_j^{(n)}$ we see appears exactly $h$ positions away. Thus, if $u^{(n)}$ is a fixed point, we have that the symbol $u^{(n)}_0$ may reappear only in positions $u^{(n)}_{kh}$ (note that, in general, in these positions we can see any element of the set $\mathcal{A}^{(n)}_j$ which contains $u^{(n)}_0$, so usually not all symbols $u^{(n)}_{kh}$ equal $u^{(n)}_0$).

Hence, the set $\{k : u^{(n)}_k = u^{(n)}_0\}$ contains only multiples of $h$; by the definition of $h^{(n)}$, this, in turn, implies that $h\mid h^{(n)}$. If $\lambda=\ee^{2\pi\ii/m}$ is a continuous eigenvalue of $X_{\boldsymbol{ \theta}}$ which is coprime to all $q_n$'s, we must have that $\lambda = \ee^{2\pi\ii r/h}$ for some $r$,  by the maximality of $h$, and thus $m$ must divide $h$. This implies that $m\mid h^{(n)}$.
\end{proof}

The following result is the generalisation of \cite[Lem.~10]{Dekking1977}. 
Let $h(\boldsymbol{\theta})$ be the  height given in Definition~\ref{def:dynamical-height}, and let $h_{\rm comb}(\boldsymbol{\theta})$ denote the combinatorial height defined in Eq.~\eqref{eq:height-combin}.
\begin{theorem}\label{thm:combinatorial-height}
Let $\boldsymbol{\theta}$ be a  torsion-free directive sequence defined on a se\-quence of bounded alphabets, with length-sequence $(q_n)$.    Then 
\[
h_{\rm comb}(\boldsymbol{\theta})= h({\boldsymbol \theta}).
\]
\end{theorem}
\begin{proof}

As discussed in the proof of Lemma \ref{lem:EV-div}, we can assume that $\boldsymbol\theta$ is injective. 
If $h_{\rm comb}(\boldsymbol{\theta})=h^{(k)}$,  then Lemma  \ref{lem:comb-height-eigenvalue} tells us that $X^{(0)}$ has a $\sigma^{h^{(k)}}$-cyclic partition into $h^{(k)}$ sets, that is,
  $\gamma(h^{(k)}) = h^{(k)} $.  Thus for any $n$, $h^{(n)}$  divides $ h({\boldsymbol \theta})$; in particular, $h_{\rm comb}(\boldsymbol{\theta})\mid h(\boldsymbol{ \theta})$.

As $\boldsymbol{\theta}$ is torsion-free, $\ee^{2\pi\ii/h(\boldsymbol{ \theta})}$ is a continuous eigenvalue satisfying the hypothesis of Lemma~\ref{lem:EV-div}; thus, for all sufficiently large $n$, we must have $h(\boldsymbol{ \theta})\mid h^{(n)}$. Hence, $h(\boldsymbol{ \theta})\le h_{\rm comb}(\boldsymbol{ \theta})$. Together with the previous observation, this gives the desired equality.
\end{proof}

\section{The column number of a constant-length directive sequence}\label{sec:column-number}

	We propose a candidate for the {\em column number} $c({\boldsymbol{\theta}})$ of a constant-length directive sequence, which generalises the definition of the column number for constant-length substitutions.  It also develops preliminary notions of a directive sequence having a coincidence that were discussed in \cite[Section 6]{BCY-2022}. If $\boldsymbol{\theta}$ has length sequence $(q_n)$, and  provided that our directive sequence is quasi-recognizable, we prove that the chosen definition ensures that the maximal equicontinuous factor map is at least $c({\boldsymbol{\theta}})$-to-$1$ and that the fibre cardinality is exactly $c({\boldsymbol{\theta}})$ for at least one orbit. As an application we use the column number to make statements about the nature of the maximal spectral type of these systems  in Section \ref{sec:spectrum}. 
	
	We work with quasi-recognizable directive sequences. To define the col\-umn number $c({\boldsymbol{\theta}})$, we first work with the tiling factor map $\pi_{\rm tile}\colon X_{\boldsymbol{ \theta}}\to\Z_{(q_n)} $ that quasi-recognizability   guarantees. As Corollary \ref{cor:MEF-torsion-free} tells us, the tiling factor map is not necessarily a maximal equicontinuous factor map, but with it we can define a {\em naïve} column number. This is an intermediate step which already gives us the column number for directive sequences with trivial height. We then show, in Theorem \ref{cor:fibre-card-column-num}, that the correct notion of column number of ${\boldsymbol{\theta}}$ 
	is simply that of its pure base as defined in Section \ref{subsec:pure-base}, and which has trivial height by Theorem \ref{thm:height}.

		As in  Section \ref{sec:combinatorial-height}, we will need recognizability. Given a shift space $X_{\boldsymbol \theta}$ generated by a quasi-recognizable directive sequence, we use the  recognizable directive sequence $\boldsymbol {\widehat\theta}$ such that $X_{\boldsymbol \theta}= X_{\boldsymbol {\widehat\theta}}$, guaranteed by Theorem \ref{thm:sadic-recognizable}.
		In what follows we shall make use of this assumption whenever it is convenient to do so.

		If $\theta\colon \mathcal A\rightarrow \mathcal B^{+}$ has length $\ell$, one can describe it using  $\ell$ maps $\theta_i\colon\mathcal A \rightarrow \mathcal B$, $0\leq i \leq \ell-1$, where
\begin{equation}\label{eq-as-perm}
\theta(a) = \theta_0(a)\cdots \theta_{\ell-1}(a) 
\end{equation}
for each $a\in\mathcal A$. We call each $\theta_j$ a {\em column} of $\theta$.

		Let ${\boldsymbol \theta}$ be a constant-length injective directive sequence on a sequence of alphabets $(\mathcal A_n)_{n\ge 0}$ of bounded size, with length sequence $(q_n)_{n\geq 0}$.
		  For $m\geq 0$, and using the notion of the columns in  \eqref{eq-as-perm}, we define 
		  		\[c (\boldsymbol{\theta},m)\coloneqq \inf_{n>m} \left\{\lvert (\theta^{[m,n)})_j(\mathcal{A}_n)\rvert  : 0\le j <  \frac{p_{n}}{p_{m}} \right\}.\] 
				\begin{definition}[Column number]\label{def:column-number}
	Noting that $(c (\boldsymbol{\theta},m))_{m\geq 0}$ is an increasing and bounded sequence, we define the \emph{na\"ive column number}  $\bar{c}({\boldsymbol{\theta}})$ to be
				\[
			\bar{c}({\boldsymbol{\theta}}) \coloneqq \lim_{m\to\infty}  c (\boldsymbol{\theta},m) =  \max_{m\geq 0}  c (\boldsymbol{\theta},m),		\]
		that is, the least cardinality of a column that appears in  $\theta^{[m,n)}$ for some $n>m$, as $n$ tends to infinity. For a non-injective directive sequence, we define $\bar{c}(\boldsymbol{ \theta})\coloneqq\bar{c}(\boldsymbol{\widehat{\theta}})$, that is, the column number of its corresponding injectivisation as given by Theorem \ref{thm:sadic-recognizable}. Let $\boldsymbol{\bar\theta}$ be the pure base of $\boldsymbol{\theta}$. We define the (real) \emph{column number} of the directive sequence $\boldsymbol{ \theta}$, $c(\boldsymbol{ \theta})$, to be the na\"ive column number of its pure base $\boldsymbol{\bar{\theta}}$.
	\end{definition}

		If $\boldsymbol{\theta} = (\theta, \theta, \dotsc)$ is a stationary directive sequence, then  the definition of $c(\theta,m)$ does not depend on $m$, as $\theta^{[m,n)}=\theta^{n-m}$. Also, it equals the definition of the column number for a single substitution, as the least cardinality that appears in some column upon iteration of the pure base $\bar\theta$. This shows that the column number $c(\theta)$ is a direct generalisation of the original notion. Note that the additional injectivity hypothesis does not make a difference in this particular context, as every primitive substitutive subshift is conjugate to one given by an injective substitution, and $c(\theta)$ is a conjugacy invariant, as it only depends on the maximal equicontinuous factor.

	The na\"ive column number $\bar{c}(  {\boldsymbol{\theta}}  )$ is finite and bounded by $\max_{n\ge 0}\lvert\mathcal{A}_n\rvert < \infty$, as every column in $\theta^{[m,n)}$ cannot have more than $\lvert\mathcal{A}_n\rvert$ different symbols. A similar bound immediately follows for $c(\theta)$. 
		 The column number can be equally defined for constant-length directive sequences defined on a sequence $(\mathcal A_n)$ of alphabets of unbounded size, but in this case, it may not be bounded.
	As column cardinalities are integers, then for a fixed value of $m$, there is an $n$ and $j$ such that $(\theta^{[m,n)})_j$ has cardinality $c({\boldsymbol{\theta}},m)$, and this cardinality is achieved as a column cardinality of $ \theta^{[m,n')}$  for all  $n'>n$. As $\bar{c}({\boldsymbol{\theta}})=c({\boldsymbol{\theta},m_0})$ for a sufficiently large $m_0$, we have:
	\begin{enumerate}[label=(\arabic*)]
		\item for any sufficiently large $n$, the morphism $\theta^{[m_0,n)}$ has at least one column with cardinality $\bar{c}({\boldsymbol{\theta}})$ and all columns have cardinality at least $\bar{c}({\boldsymbol{\theta}})$, and
		\item for any $m>m_0$ we may find some $n$ such that $\theta^{[m,n)}$ has a column with cardinality $\bar{c}({\boldsymbol{\theta}})$. 
	\end{enumerate}

	Let ${\boldsymbol \theta}$ be quasi-recognizable with $\pi_{\rm tile}\colon X_{\boldsymbol \theta} \rightarrow \Z_{(q_n)}$ its associated tiling factor map; note that, since we are assuming that $\boldsymbol{ \theta}$ is injective, this implies recognisability. A $\pi_{\rm tile}$-fibre $\pi_{\rm tile}^{-1}(z)$ is called {\em regular} if it has 
	minimal car\-di\-nal\-i\-ty.
	We will prove that  the regular fibres of the factor map $\pi_{\rm tile}$ have cardinality exactly $c({\boldsymbol{\theta}})$. We split the proof into three small lemmas.
	\begin{lemma}\label{lem:column_number_upperbound}
		Let $\boldsymbol{\theta}$ be a quasi-recognizable  directive sequence with length sequence $(q_n)$, 
		defined on a sequence of bounded alphabets. Suppose  that its na\"ive column number is given by $\bar{c}({\boldsymbol{\theta}})=c({\boldsymbol{\theta},0})$. Then, there is some element $z\in\mathbb{Z}_{(q_n)}$ such that $\pi_{\rm tile}^{-1}(z)$ has exactly $\bar{c}({\boldsymbol{\theta}})$ elements.	\end{lemma}

	\begin{proof}
		For any $n\ge 1$, any point $x\in X_{\boldsymbol{\theta}}$ is a concatenation of $n_0$-th order supertiles $\theta^{[0,n_0)}(a)$ for some $a\in\mathcal{A}_{n_0}$ \cite[Lemma 4.2]{BSTY-2019}. Fix $n_0$.  As $\bar{c}({\boldsymbol{\theta}})=c({\boldsymbol{\theta},0})\le c({\boldsymbol{\theta},n_0})\le \bar{c}({\boldsymbol{\theta}})$,  there must be some $n_1>n_0$ such that $\theta^{[n_0,n_1)}$ has a column with cardinality $\bar{c}({\boldsymbol{\theta}})$; by taking a larger $n_1$ if needed, we can ensure that such a column is neither the first nor the last column of the morphism $\theta^{[n_0,n_1)}$. Let $0 < j_1 < \left(p_{n_1}/p_{n_0}\right) - 1 $ be the index of this column.
		
		The $n_1$-supertiles are concatenations of $\left(p_{n_1}/p_{n_0}\right)$ $n_0$-supertiles. This implies that the $j_1$-st of these $n_0$-supertiles equals $\theta^{[0,n_0)}(a)$ with $a\in M\subseteq\mathcal{A}_0, \lvert M\rvert=\bar{c}({\boldsymbol{\theta}})$. Furthermore, as no column of $\theta^{[0,n_0)}$ can have cardinality less than $\bar{c}({\boldsymbol{\theta}})$, the restriction of $\theta^{[0,n_0)}$ to $M$ is injective. That is, any point $x\in X_{\boldsymbol{\theta}}$ such that $\pi_{\rm tile}(x)\equiv p_{n_0}j_1 \pmod{p_{n_1}}$ has exactly one of $\bar{c}({\boldsymbol{\theta}})$ different supertiles with support $[0,p_{n_0})$

		We can iterate this process, and find some $n_2>n_1$ such that $\theta^{[n_1,n_2)}$  has a column, with index $j_{2}$, such that it has cardinality $\bar{c}({\boldsymbol{\theta}})$; this is possible as a consequence of property (2) stated above. Once again we may assume that $0 < j_{2} < \left(p_{n_2}/p_{n_1}\right) - 1$, i.e. this is neither the first nor the last column. Thus, every point in $X_{\boldsymbol{\theta}}$ is a concatenation of $n_2$-supertiles $\theta^{[0,n_2)}(a)$, each of which is a concatenation of $n_1$-supertiles. Also, the $j_{2}$-th of these $n_1$-supertiles is of the form $\theta^{[0,n_1)}(a)$ for some $a\in M'\subseteq\mathcal{A}_{n_1}$, with $M'$ of cardinality $c({\boldsymbol{\theta}})$ by the same argument as above.
		
		Thus, if we have some $x\in X_{\boldsymbol{\theta}}$ such that $\pi_{\rm tile}(x) \equiv p_{n_1}j_{2}+p_{n_0}j_{1}\pmod{p_{n_2}}$, the $n_1$-th order supertile of $x$ passing through the origin is one of $\bar{c}({\boldsymbol{\theta}})$ possible options, and its $j_{1}$-th component $n_0$-supertile is one of $\bar{c}({\boldsymbol{\theta}})$ different possible options as well. As no columns with cardinality less than $\bar{c}({\boldsymbol{\theta}})$ appear, there is a bijection, induced by the $j_{1}$-th column of the morphism $\theta^{[n_0,n_1)}$, between the $\bar{c}({\boldsymbol{\theta}})$ possible options for the $n_0$-th order supertile of $x$ at $[0,p_{n_0})$ and the $\bar{c}({\boldsymbol{\theta}})$ possible options for the $n_1$-th order supertile of $x$ that passes through the origin. Note that the support of this $n_1$-th order supertile contains both positive and negative integers, and after future iterations of the same process, the support of the $n_k$-supertile obtained by this process grows to $(-\infty, \infty)$ as $k\to\infty$.
		
		We iterate the above procedure, converging to an  infinite sum $z:=\sum_i p_{n_i}j_{i+1}\in \mathbb{Z}_{(q_n)}$. The above argument shows, in summary, that:
		\begin{itemize}
			\item there exists a sequence of supertiles $w_1,w_2,\dotsc$ of increasing size such that each supertile $w_k$ is one of the component supertiles of $w_{k+1}$, and $w_k$ determines $w_{k+1}$ uniquely,
			\item there are $\bar{c}({\boldsymbol{\theta}})$ possible options for $w_1$, and thus there exist exactly $\bar{c}({\boldsymbol{\theta}})$ possible sequences,
			\item if $\pi_{\rm tile}(x)=z$, the supertile of corresponding size that passes through the origin is forced to be one of the $w_k$, and
			\item by the choice of $z$, the support of $w_k$ in $x$ grows to infinity in both directions as $k\to\infty$.
		\end{itemize}
		Thus, each of the $\bar{c}({\boldsymbol{\theta}})$ choices for $w_1$ determines $x$ entirely, hence there are only $\bar{c}({\boldsymbol{\theta}})$ possible elements of $X_{\boldsymbol{\theta}}$ for which $\pi_{\rm tile}(x)=z$, as desired.
	\end{proof}

	\begin{lemma}\label{lem:column_num_lowerbound}
	Let $\boldsymbol{\theta}$ be a quasi-recognizable  directive sequence with constant-length sequence $(q_n)$, defined on a sequence of bounded alphabets.  Then
	any fibre $\pi_{\rm tile}^{-1}(z)$ of the tiling factor map has at least  $c({\boldsymbol{\theta},0})$ elements.	\end{lemma}

	\begin{proof}
		Given $z\in \Z_{(q_n)}$, there is a sequence of integers $(Z_j)$ such that $Z_{j+1}\equiv Z_j \pmod{p_j}$ and $Z_j\rightarrow z$.
			Given $a\in \mathcal A_{j+1}$,	
		let $x^{(a,j)}$ be some element of $X_{\boldsymbol{\theta}}$ such that $x^{(a,j)}\rvert_{[-Z_j,p_j-Z_j-1]} = \theta^{[0,j+1)}(a)$, that is, such that its central supertile is $\theta^{[0,j+1)}(a)$ and the $Z_j$-th column of this supertile is at the origin. 
				By the definition of $c({\boldsymbol{\theta}},0)$, the set $U_j =  \{x^{(a,j)}  :  a\in\mathcal{A}_{j+1} \}$ contains at least $c({\boldsymbol{\theta}},0)$
		 different elements, which differ pairwise in their $0$-th coordinate. Thus, we may partition $U_j$ into $c({\boldsymbol{\theta}},0)$ or more disjoint sets $U_{j,b}$ with $b\in\mathcal{A}_0$, given by    $x\in U_{j,b}$ if     $x_0=b$.
				
		By the fact that  $Z_{j+1}\equiv Z_j \pmod{p_j}$, if $U_{j,b}$ is non-empty, then $U_{k,b}\ne\varnothing$ for any $k<j$ as well. Thus, there exists some set $\mathcal{B}\subseteq\mathcal{A}_0$ with $\lvert\mathcal{B}\rvert\ge c(\boldsymbol{\theta},0)$ such that $U_{j,b}$ is non-empty for every value of $j$ and every $b\in\mathcal{B}$. For every $b\in\mathcal{B}$, we may find an accumulation point $x^{(b)}$   of some sequence $y^{(j)}\in X_{\boldsymbol{\theta}}$ with $y^{(j)}\in U_{j,b}$; and $x^{(b)}_0=b$ so that the $x^{(b)}$'s are all distinct and there are at least $c(\boldsymbol{\theta},0)$ of them.
		
		Also, by the choice of $Z_j$ and the condition $Z_j\equiv Z_{j+1}\pmod{p_j}$, each $x^{(b)}$ is guaranteed to have a $j$-th order supertile with support $[-Z_j,p_j-Z_j-1]$. Hence, $\pi_{\rm tile}(x^{(b)}) \equiv  Z_j \pmod{p_j}$, by choice; this naturally implies that $\pi_{\rm tile}(x^{(b)})=z$, and since we have $\lvert\mathcal{B}\rvert\ge c(\boldsymbol{\theta},1)$ different elements whose image is $z$, the desired conclusion holds.
	\end{proof}

	Thus, if the na\"ive column number $\bar{c}(\boldsymbol{\theta})$ equals $c(\boldsymbol{\theta},0)$, we can guarantee that the tiling factor map  $\pi_{\rm tile}\colon X_{\boldsymbol{\theta}} \rightarrow \mathbb{Z}_{(q_n)}$  is somewhere $\bar{c}({\boldsymbol{\theta}})$-to-$1$ and $\bar{c}({\boldsymbol{\theta}})$ is the smallest possible value of $k$ for which $\pi_{\rm tile}$ is $k$-to-$1$. The following simple result generalises this to a large collection of directive sequences:
	
	\begin{theorem}\label{thm:odom-factor-k}
	Let $\boldsymbol{\theta}$ be an injective  quasi-recognizable  directive sequence with constant-length sequence $(q_n)$, defined on a sequence of bounded al\-pha\-bets. Then the tiling factor map $\pi_{\rm tile}\colon X_{\boldsymbol{\theta}} \rightarrow \mathbb{Z}_{(q_n)_{n\geq 0}}$ is somewhere $\bar{c}({\boldsymbol{\theta}})$-to-$1$, and it is not $k$-to-$1$ anywhere for any $k<\bar{c}({\boldsymbol{\theta}})$.
	\end{theorem}

	\begin{proof}
		Let $n$ be the smallest integer such that $\bar{c}({\boldsymbol{\theta}})=c({\boldsymbol{\theta},n})$, and consider the directive sequence
		 $\boldsymbol{\vartheta}=(\theta^{(n)},\theta^{(n+1)},\dotsc)$
		 obtained from $\boldsymbol{\theta}$ by removing the first $n$ morphisms.
		  Then $\bar{c}(\boldsymbol{\vartheta})=c(\boldsymbol{\vartheta},0)$ and $\boldsymbol{\vartheta}$ is quasi-recognizable, so the factor map $\pi'\colon X_{\boldsymbol{\vartheta}} \rightarrow \mathbb{Z}_{(q_j)_{j\geq n}}$
		cannot be less than $\bar{c}(\boldsymbol{\vartheta})$-to-1 by Lemma~\ref{lem:column_num_lowerbound}, and the existence of a fibre with cardinality $\bar{c}(\boldsymbol{\vartheta})$ is guaranteed by Lemma~\ref{lem:column_number_upperbound}. As $c(\boldsymbol{\theta},m+n)=c(\boldsymbol{\vartheta},m)$ and the sequence $c(\boldsymbol{\theta},k)$ is increasing, it is not hard to see that $\bar{c}({\boldsymbol{\theta}})=\bar{c}(\boldsymbol{\vartheta})$.
		
		Note that every point $x$ in $X^{(n)}$ gives birth to $q_0\cdots q_{n-1}$ points in $X_{\boldsymbol{\theta}}$, namely 
		$\sigma^j\theta^{[0,n)}(x)$, $0\leq j\leq p_{n-1}-1$, which are all distinct since $\theta^{[0,n)}$ is injective. Let $\bar \pi\colon  X^{(n)}\rightarrow \Z_{(q_j)_{j\geq n}}$ and let $\pi_{\rm tile}\colon X_{\boldsymbol \theta}\rightarrow \Z_{(q_j)_{j\geq 0}}$ be the respective factor maps. If $\bar\pi(x)=z $, then $\pi_{\rm tile} (   \sigma^j\theta^{[0,n)}(x)   )= zw_j$ where $w_j$ is the expansion of $j$ with respect to the base $(p_n)_{n\geq 0}$. If $\lvert\bar\pi^{-1}(z)\rvert=c$, injectivity of $\theta^{[0,n)}$ implies that $\lvert\pi_{\rm tile}^{-1}(z0^n)\rvert= \lvert\bar\pi^{-1}(z)\rvert=c$, i.e. that all fibres have cardinality at least $\bar{c}(  \boldsymbol{\theta}   )$. 
	\end{proof}

	As the na\"ive column number $\bar{c}(\boldsymbol{ \theta})$ is defined in terms of the injectivisation of $\boldsymbol{ \theta}$, the following is an immediate consequence:

	\begin{corollary}\label{cor:odom-factor-k}
		Let $\boldsymbol{\theta}$ be a torsion-free directive sequence defined on a sequence of bounded alphabets.
		Then its tiling factor map 
		 is somewhere $\bar{c}({\boldsymbol{\theta}})$-to-$1$, and it is not $k$-to-$1$ anywhere for any $k<\bar{c}({\boldsymbol{\theta}})$.
\end{corollary}
	
	As noted above, all the above results relate the na\"ive column number $\bar{c}(\boldsymbol{ \theta})$ with the fibres of the tiling factor map over the odometer $\mathbb{Z}_{(q_n)}$ determined by the lengths of the morphisms in the directive sequence $\boldsymbol{ \theta}$. 
	  Despite the fact that       the na\"ive column number does not necessarily convey information about the fibres of $\pi_{\rm MEF}$, in what follows      
	  we show that the ``true'' column number $c(\boldsymbol{ \theta})=\bar{c}(\boldsymbol{ \bar  \theta})$ has the same relationship with the cardinality of the fibres of $\pi_{\rm MEF}$ as the tiling factor map onto $\mathbb{Z}_{(q_n)}$ fibres with the na\"ive column number $\bar{c}(\boldsymbol{ \theta})$.
		
	\begin{corollary}
		\label{cor:fibre-card-column-num}
			Let $\boldsymbol{\theta}$ be a torsion-free directive sequence defined on a sequence of bounded alphabets, 
						and let $\pi_{\rm MEF}\colon X_{\boldsymbol{\theta}} \rightarrow \mathbb{Z}_{(q_n)_{n\geq 0}}\times \Z/h\Z$ be its maximal equi\-con\-ti\-nuous factor map, where $h$ is the height of the directive sequence $\boldsymbol{ \theta}$. Then $\pi_{\rm MEF}$ is somewhere $c({\boldsymbol{\theta}})$-to-$1$, and it is not $k$-to-$1$ anywhere for any $k<c({\boldsymbol{\theta}})$.
	\end{corollary}
	
	\begin{proof}
		If $\boldsymbol{ \theta}$ has trivial height, this follows from Corollaries~\ref{cor:MEF-torsion-free}~and~\ref{cor:odom-factor-k}. 
		More generally, we can write an explicit form for the maximal equicontinuous factor map $\pi_{\rm MEF}\colon X_{\boldsymbol{ \theta}}\to\Z_{(q_n)}\times \Z/h\Z$ in terms of the tiling factor map of its pure base, that is, the pure base's MEF map $\bar{\pi}_{\rm MEF}\colon X_{\boldsymbol{ \bar  \theta}}\to \Z_{(q_n)}$, as follows. Recall the map $\varphi\colon X_{\boldsymbol{ \theta}}\to X_{\boldsymbol{ \bar\theta}}\times \Z/h\Z$ that is the conjugacy between $(X_{\boldsymbol{ \theta}},\sigma)$ and the suspension $(X_{\boldsymbol{ \theta}},T)$, given by Theorem~\ref{thm:height}. Then we can take
			\[\pi_{\rm MEF}(x) = \left(h\cdot\bar{\pi}_{\rm MEF}(\bar{x}) + m_x,m_x\right),\text{ where }\varphi(x) = (\bar{x},m_x).\]
		 We are implicitly using the fact that $h$ is coprime to all $q_n$'s, as the above definition relies on the equality $h\cdot\Z_{(q_n)} = \Z_{(q_n)}$, which is  true as the integer $h$ has a multiplicative inverse in $\Z_{(q_n)}$.
Then the preimage of  $(z,m)\in \Z_{(q_n)}\times \Z/h\Z$ under $\pi_{\rm MEF}$ as defined above is given by:
			\begin{align*} x=\varphi(\bar{x},m_x)\in\pi_{\rm MEF}^{-1}(z,m)&\iff h \cdot\bar{\pi}_{\rm MEF}(\bar{x}) + m_x = z \mbox{ and } m = m_x\\ & \iff \bar{\pi}_{\rm MEF}(\bar{x}) = \frac{z-m}{h}.\end{align*}
				We see that $m$ determines $m_x$ entirely, so the set of possible $x$ is in a 1-1 correspondence with the set of possible $\bar{x}$, which belong in the fibre $\bar{\pi}_{\rm MEF}^{-1}((z-m)/h)$. By Corollary~\ref{cor:odom-factor-k}, this fibre cannot have less than $\bar{c}(\boldsymbol{ \bar  \theta})=c(\boldsymbol{ \theta})$ elements, so the same holds for the fibres of $\pi_{\rm MEF}$. Also, there exists some $z^*\in\Z_{(q_n)}$ such that $\lvert \bar{\pi}_{\rm MEF}^{-1}(z^*) \rvert = c(\boldsymbol{ \theta})$; choosing an arbitrary $m$ and taking $z = h\cdot z^* + m$ gives $(z,m)$ as an element of $\Z_{(q_n)}\times \Z/h\Z$ with exactly $c(\boldsymbol{ \theta})$ preimages under $\pi_{\rm MEF}$, as desired.
	\end{proof}

	\begin{remark}
		In the proofs of Lemmas~\ref{lem:column_number_upperbound}~and~\ref{lem:column_num_lowerbound}, the injectivity hypothesis in the definition of $\bar{c}(\boldsymbol{ \theta})$ is not used at all, and the proof of Theorem~\ref{thm:odom-factor-k} only uses injectivity up to level $n$. This allows us to give an estimate on fibre cardinalities from any directive sequence, not necessarily injective. However, the injectivity property in the definition of both the na\"ive and true column numbers is essential to get sharp bounds as in Corollaries~\ref{cor:odom-factor-k}~and~\ref{cor:fibre-card-column-num}; 
		as the following example shows.
	\end{remark}

	\begin{example}\label{ex:naive_col_num_noninjective}
		Consider the directive sequence $\boldsymbol{ \alpha}=(\varrho,\vartheta,\vartheta,\dotsc)$ from Examples~\ref{ex:quasi_recognizable_but_not_recognizable}~and~\ref{ex:injectivisation-2}. It is not hard to verify that $c(\boldsymbol{ \alpha},m)=2$ for any $m\ge 1$, which is a consequence of the fact that $X_\vartheta$ has height $2$; thus, if we ignore the injectivity hypothesis in the definition, we could say that its na\"ive column number equals $\bar{c}(\boldsymbol{ \alpha})=2$.
		
		However, as $X_{\boldsymbol{ \alpha}}=X_\theta$ has a coincidence (and thus has column number $1$ in the classical sense, being an almost 1-1 extension of the underlying odometer), the lower bound from Theorem~\ref{thm:odom-factor-k} does not apply in this situation. 
		The bounds given by Lemmas~\ref{lem:column_number_upperbound}~and~\ref{lem:column_num_lowerbound}  still apply to the directive sequence $\boldsymbol{ \alpha}$, and  moving from $\boldsymbol{ \alpha}$ to its injectivisation $\boldsymbol{ \widehat\alpha} = (\theta,\theta,\dotsc)$ gives us the correct result as in Corollary~\ref{cor:odom-factor-k}.
	\end{example}

	\color{black}

We end this section with the question: Does   the height $h(\boldsymbol{ \theta})$ always divide the naïve column number $\bar{c}(\boldsymbol{ \theta})$? In the case where  $\boldsymbol{ \theta}$ is stationary, Lema\'{n}czyk and M\"{u}llner showed that this is true 
\cite[Lemma 2.3]{Lemanczyk-Muellner}. An investigation of their proof suggests that this seems to be the case, at least when the corresponding $S$-adic shift is uniquely ergodic.

\section{Mixed spectrum and discontinuous eigenvalues} \label{sec:spectrum}

In this section, we use the concept of a column number that we introduced in Section \ref{sec:column-number} to identify finer properties of the  spectrum of the measure-theoretic dynamical system  $(X_{\boldsymbol{\theta}},\sigma,\mu)$, where  $\mu$ is a $\sigma$-invariant measure on $X_{\boldsymbol{\theta}}$. 
Consider the Hilbert space $L^{2}(X,\mu)$. Let $H_d\subseteq L^{2}(X,\mu)$ be the closure of the span of all measurable eigenfunctions. We say that  $(X_{\boldsymbol{\theta}},\sigma,\mu)$ has {\em discrete spectrum (pure point spectrum)} if $L^{2}(X,\mu)=H_d$. Otherwise, we say that  $(X_{\boldsymbol{\theta}},\sigma,\mu)$  has {\em mixed spectrum}, i.e., the unitary operator $U_{\sigma}\colon f\mapsto f\circ \sigma$ admits both discrete and continuous spectral components.

We begin by recalling Dekking's result relating continuous spectrum and the column number
of a constant-length substitution. Recall that aperiodic primitive substitutions are uniquely
ergodic \cite{Michel-1976}.	
As in Section~\ref{sec:column-number}, $\theta$ here corresponds to the stationary directive sequence $\boldsymbol{\theta}=(\theta,\theta,\ldots)$ and $c(\theta)=c(\boldsymbol{\theta})$.

\begin{theorem}[\hspace{1sp}{\cite[Theorem~7]{Dekking1977}}]
Let $\theta$ be an aperiodic primitive constant-length substitution. Then $(X_{\theta},\sigma,\mu)$ has discrete spectrum if and only if $c(\theta)=1$.  Otherwise, it has mixed spectrum.
\end{theorem}

Discrete spectrum, which is a measure-theoretic property, is connected to the notion of mean equicontinuity for topological dynamical systems; see \cite{ABKL-2015, FuhrmannGrogerLenz2022}. Let $(X,\sigma)$ be a topological dynamical system, where $X$ is a compact metric space and $\sigma$ a homeomorphism, so that it defines a $\Z$-action on $X$. The {\em Weyl metric} $d_{\textnormal{W}}$ is defined on $X$ as 
\[
d_{\textnormal{W}}(x,y)\coloneqq \limsup_{n-m\to\infty}\frac{1}{n-m}\sum_{i=m}^{n-1}d(\sigma^{i}x,\sigma^{i}y)
\]
where $d$ is the metric on $X$. The system is $(X,\sigma)$ is called \emph{Weyl mean equicontinous} or just \emph{mean equicontinuous} if for every $\varepsilon>0$, there exists a $\delta>0$ such that 
$d(x,y)<\delta$ implies $d_{\textnormal{W}}(x,y)<\varepsilon$.
This is a generalisation of the notion of equicontinuity mentioned in Section~\ref{sec:spectrum-background}.

Let $(X,G)$ and $(Y,G)$ be topological dynamical systems. Let $\pi\colon X\to Y$ be a factor map and, for a $G$-invariant measure $\mu$ on $X$, let $\pi(\mu)$ be the corresponding pushforward measure on $Y$. 
We call $(X,G)$ a \emph{regular extension} of $(Y,G)$ if 
\begin{equation}\label{eq:regular-ext}
\pi(\mu)\left(\big\{ y\in \pi(X) : \lvert\pi^{-1}(y)\rvert=1\big\}\right)=1
\end{equation}
for any $G$-invariant measure $\mu$ on $X$. Otherwise, $(X,G)$ is called an \emph{irregular extension} of $(Y,G)$. 

Recall that two points $x,\bar x$ are {\em proximal} if there is a sequence $(g_k)$ of group elements such that $d (g_k( x), g_k(\bar x) )\rightarrow 0$ as $k\rightarrow \infty$. Otherwise they are {\em distal}.
Define the \emph{minimal rank} (\textbf{mr}) of a dynamical system to be the minimum cardinality of elements in a fibre over the MEF, and the 
 \emph{coincidence rank} (\textbf{cr}) to be  the maximal number of mutually distal points in a fibre.   A fibre is {\em distal} if any two points in the fibre are distal.
  If the dynamical system is minimal, then it can be seen that the \textbf{cr} is constant over different fibres, so that $\textbf{cr}\leq \textbf{mr}$.
The following result, due to  Barge and Kellendonk, relates discrete spectrum to \textbf{mr} and \textbf{cr} as follows.  (see \cite[Theorem~4.12]{ABKL-2015} and \cite[Theorem~2.25]{BK-2013}). 

\begin{theorem}\label{thm:coincidence-rank}
Let $(X,G)$ be a minimal system with finite coincidence rank. Suppose that the set of distal fibres over the MEF has full Haar measure. Let $\mu$ be an ergodic probability measure on $X$. Then the following are equivalent:
\begin{itemize}
\item $\textnormal{\textbf{cr}}=1$.
\item The system is an almost $1$-$1$ extension of its maximal equicontinuous factor.
\item The continuous eigenfunctions generate $L^{2}(X,\mu)$.
\end{itemize}
Moreover, if one of these conditions hold, then $(X,G)$ is uniquely ergodic. 
\end{theorem}

It follows from Corollary~\ref{cor:fibre-card-column-num} that 
for a torsion-free $S$-adic directive sequence $\boldsymbol{\theta}$ on a sequence of bounded alphabets, if $c(\boldsymbol{\theta})=1$, then $\textbf{cr}=\textbf{mr}=c(\boldsymbol{\theta})$. Moreover, the condition on the distal fibres having full Haar measure is equivalent to Eq.~\eqref{eq:regular-ext}. We then get the following sufficient condition for discrete spectrum for torsion free $S$-adic sequences, which is a consequence of Theorem~\ref{thm:coincidence-rank} and Corollary~\ref{cor:fibre-card-column-num}.

\begin{proposition}\label{prop:discrete-spec}
Let ${\boldsymbol \theta}$  be a 
torsion-free  directive sequence defined on a sequence of bounded alphabets. Let $\mu$ be an ergodic invariant measure on $(X_{\boldsymbol{\theta}},\sigma)$. Suppose Eq.~\eqref{eq:regular-ext} holds and that $c(\boldsymbol{\theta})=1$. Then 
\begin{itemize}
\item $(X_{\boldsymbol{\theta}},\sigma)$ is uniquely ergodic, and
\item $(X_{\boldsymbol{\theta}},\sigma,\mu)$ has discrete spectrum  with continuous eigenfunctions. 
\end{itemize}
\end{proposition}

The equivalences in Theorem~\ref{thm:coincidence-rank} are related to mean equicontinuity.  In the minimal case, we have the following result due to Fuhrmann, Gr\"{o}ger and Lenz; see \cite{FuhrmannGrogerLenz2022}. Here, $G$ is a general (sigma-compact, amenable) group acting  continuously on $X$. In our setting, $G=\mathbb{Z}$. We also refer the reader to \cite{LTY-2015,GR-2019} for weaker notions which are equivalent to discrete spectrum (without necessarily requiring all eigenfunctions to be continuous).

\begin{theorem}[\hspace{1sp}\textnormal{\cite[Cor.~1.6]{FuhrmannGrogerLenz2022}}]\label{thm:mean-ppspec}
Let $(X,G)$ be minimal. Then $(X,G)$ is mean equicontinuous if and only if {it is uniquely ergodic  and has discrete spectrum with continuous eigenfunctions.}
\end{theorem}

\begin{remark}
Unlike in the case of substitutions, the primitivity of a directive sequence is not sufficient to guarantee unique ergodicity. In \cite{FerencziFisherTalet2009}, a condition for unique ergodicity was given for a  minimal constant-length directive sequence $\boldsymbol{\theta}=(\theta^{(0)},\theta^{(1)},\ldots)$ in the case where the substitution matrix $M_i$ of $\theta^{(i)}$ is the symmetric matrix
$M_i=\begin{pmatrix}
1 & n_i\\
n_i & 1
\end{pmatrix}$, where $n_i\in\mathbb{N}$. The authors show that 
 the shift $(X_{\boldsymbol{\theta}},\sigma)$ is uniquely ergodic if and only if $\sum_i \frac{1}{n_i}=\infty$; see \cite[Prop.~3.1]{FerencziFisherTalet2009} and more generally \cite[Example 5.5]{BKY-2014}.  Both these references work with Bratteli-Vershik systems, but as substitutions on a two letter alphabet are recognizable, $(X_{\boldsymbol{\theta}},\sigma)$ is  almost-conjugate to the corresponding Bratteli-Vershik system \cite[Theorems 4.6 and  6.5]{BSTY-2019}.\end{remark}	

A {\em  Toeplitz shift} is a  shift $(X,\sigma)$, $X\subset \mathcal A^{\Z}$ with $\mathcal A$ finite, which is an almost automorphic extension of an odometer and hence minimal. A  Toeplitz shift al\-ways has a representation as a quasi-recognizable $S$-adic shift \cite[Theorem~8]{GjerdeJohansen2000}.
Note that they always have column number  1, by Lemma~\ref{lem:column_num_lowerbound}. We also briefly remark that the Toeplitz shifts which satisfy the regularity condition in Eq.~\ref{eq:regular-ext} coincide with the family of \emph{regular Toeplitz shifts}, as presented in \cite{Williams} and \cite{DownarowiczLacroix1998}. 

In certain instances, e.g., when the system does not admit a fibre of infinite cardinality over the MEF, the failure of mean equicontinuity can be traced back to  Eq.~\eqref{eq:regular-ext}.

\begin{theorem}[\hspace{-0.06em}{\cite[Cor.~7.7]{FuhrmannGrogerLenz2022}}]\label{thm:mean-irreg}
Suppose that $(X,G)$ is an irregular  extension of $(Y,G)$ via the factor map $\pi\colon X\to Y$ and suppose that $(X,G)$ and $(Y,G)$ have the same maximal e\-qui\-con\-tin\-u\-ous factor.
If the fibres of $\pi$ are finite then $(X,G)$ cannot be mean equicontinuous. 
\end{theorem}

The results above give us the following partial  generalisation of Dekking's result for when $c(\boldsymbol{\theta})>1$.

\begin{proposition}\label{prop:mixed-spec}
Let ${\boldsymbol \theta}$  be a 
torsion-free  directive sequence defined on a sequence of bounded alphabets. Suppose $(X_{\boldsymbol{\theta}},\sigma)$ is uniquely ergodic with unique $\sigma$-invariant measure $\mu$ and that $c(\boldsymbol{\theta})>1$. Then either
\begin{itemize}
\item $(X_{\boldsymbol{\theta}},\sigma,\mu)$ has mixed spectrum, or
\item $(X_{\boldsymbol{\theta}},\sigma,\mu)$ has discrete spectrum but admits discontinuous eigenfunctions. 
\end{itemize}
\end{proposition}

\begin{proof}Let $h$ be the height of  ${\boldsymbol{\theta}}$. 
In order to use Theorem~\ref{thm:mean-irreg}, set $(X,G)\coloneqq (X_{\boldsymbol{\theta}},\sigma)$ and $(Y,G)\coloneqq (\Z_{(q_n)}\times \Z/h\Z,+(1,1))$, where $G=\mathbb{Z}$ for both dynamical systems.
 Here, the factor map $\pi\coloneqq \pi_{\text{MEF}}$ is the same as in Theorems~\ref{thm:dekking-sadic} and \ref{thm:height}, where we work with a recognizable representation of $\boldsymbol{\theta}  $, as guaranteed by Theorem \ref{thm:sadic-recognizable}. Suppose $c(\boldsymbol{\theta})>1$. From Corollary~\ref{cor:fibre-card-column-num}, we know that $\pi_{\text{MEF}}$ is at least $c(\boldsymbol{\theta})$-to-$1$ everywhere over $\Z_{(q_n)}\times \Z/h\Z$. 
This means that the set 
\[
\big\{ (z,m)\in \Z_{(q_n)}\times \Z/h\Z :\lvert\pi^{-1}((z,m))\rvert=1\big\}
\]
is empty, and hence is always of zero $\pi_{\text{MEF}}(\mu)$ measure, where $\mu$ the unique shift-invariant measure on $X_{\boldsymbol{\theta}}$. 
This immediately implies that $(X_{\boldsymbol{\theta}},\sigma)$ is an irregular extension of $(\Z_{(q_n)}\times \Z/h\Z,+(1,1))$. Moreover, since $\max_{n\geq 0}\lvert\mathcal{A}_n\rvert<\infty$, we also have that $\lvert\pi^{-1}((z,m))\rvert<\infty $ for all $z\in \Z_{(q_n)}\times \Z/h\Z$. It then follows from Theorem~\ref{thm:mean-irreg} that $(X_{\boldsymbol{\theta}},\sigma)$ cannot be mean equicontinuous.
The minimality of $(X_{\boldsymbol{\theta}},\sigma)$ follows by the definition of torsion-freeness.
Theorem~\ref{thm:mean-ppspec} now implies that $(X_{\boldsymbol{\theta}},\sigma,\mu)$ either has mixed spectrum or has pure discrete spectrum but has eigenfunctions which are discontinuous.
\end{proof}

We note  that in the case of a stationary directive sequence, the second case does not oc\-cur, as all measurable eigenvalues are continuous  \cite{Dekking1977}. This is no longer true in the $S$-adic set\-ting; see \cite{DownarowiczLacroix1998} and  \cite[Sec.~7]{BressaudDurandMaass2010}.  An example of a non-uniquely ergodic Toeplitz  shift with discrete spectrum and which admits a discontinuous ra\-tion\-al eigenvalue can be found in \cite[Sec.~7]{BressaudDurandMaass2010}. Also, in \cite{DownarowiczLacroix1998}, the authors construct irregular extensions of odometers which are strictly ergodic and have discrete spectrum, including irrational ei\-gen\-val\-ues. In light of the evidence, i.e., Proposition \ref{prop:mixed-spec} and Proposition~\ref{prop:discrete-spec}, we conjecture that 
 if $c({\boldsymbol \theta})>1$ then it must have mixed spectrum.

  In future work we will investigate measurable eigenvalues of constant-length $S$-adic shifts.

\subsection*{Acknowledgements}
{The authors thank Michael Baake,  Gabriel Fuhrmann and Johannes Kellendonk for illuminating discussions, Christopher Cabezas for questions and a discussion which led to Example \ref{ex:quasirec-under-factor}, and Clemens M\"{u}llner for a comment relating the column number to the height. The authors also thank the reviewer for helpful comments which improved the manuscript.

This work was supported by  the
   EPSRC  grant numbers EP/V007459/2 and EP/S010335/1. The second author is supported by DAAD through a PRIME Fellowship. The first author received funding from ANID/FONDECYT Postdoctorado 3230159 during the final phase of this work.}

\bibliographystyle{amsplain}

\end{document}